\documentclass{article}

\usepackage{amstext,amsmath,amssymb}
\usepackage{epsfig,rotating,multirow}
\usepackage{algorithm,algorithmic}
\usepackage{comment}
\usepackage{geometry}
\usepackage{booktabs}
\usepackage{placeins}
\usepackage{ulem}

\usepackage{graphicx,bm,xcolor}
\usepackage{subcaption}
\usepackage{float}
\usepackage[utf8]{inputenc}
\usepackage[T1]{fontenc}

\usepackage{amsmath, amssymb, amsthm}

\newtheorem{theorem}{Theorem}[section]
\newtheorem{lemma}[theorem]{Lemma}

\numberwithin{equation}{section}

\newcommand{\stkout}[1]{\ifmmode\text{\sout{\ensuremath{#1}}}\else\sout{#1}\fi}

\def\CC{{\mathchoice {\setbox0=\hbox{$\displaystyle\rm C$}\hbox{\hbox
to0pt{\kern0.4\wd0\vrule height0.9\ht0\hss}\box0}}
{\setbox0=\hbox{$\textstyle\rm C$}\hbox{\hbox
to0pt{\kern0.4\wd0\vrule height0.9\ht0\hss}\box0}}
{\setbox0=\hbox{$\scriptstyle\rm C$}\hbox{\hbox
to0pt{\kern0.4\wd0\vrule height0.9\ht0\hss}\box0}}
{\setbox0=\hbox{$\scriptscriptstyle\rm C$}\hbox{\hbox
to0pt{\kern0.4\wd0\vrule height0.9\ht0\hss}\box0}}}}

\def\ZZ{{\mathchoice {\hbox{$\sf\textstyle Z\kern-0.4em Z$}}
{\hbox{$\sf\textstyle Z\kern-0.4em Z$}}
{\hbox{$\sf\scriptstyle Z\kern-0.3em Z$}}
{\hbox{$\sf\scriptscriptstyle Z\kern-0.2em Z$}}}}

\newcommand{\lw}[1]{\smash{\lower2.ex\hbox{#1}}}

\definecolor{gray}{gray}{0.6}

\newcommand{\vertiii}[1]{{\left\vert\kern-0.25ex\left\vert\kern-0.25ex\left\vert #1 
		\right\vert\kern-0.25ex\right\vert\kern-0.25ex\right\vert}}

\def\wtb{\widetilde{b}}

\def\wtt{\widetilde{t}}

\def\wtD{\widetilde{D}}

\def\wtT{\widetilde{T}}

\begin{document}

\title{Explicit inverse of symmetric, tridiagonal near Toeplitz matrices with strictly diagonally dominant Toeplitz part}

\author{Bakytzhan Kurmanbek\thanks{Nazarbayev University, Department of Mathematics, 53 Kabanbay Batyr Ave, Astana 010000, Kazakhstan (bakytzhan.kurmanbek@nu.edu.kz)} \and Yogi Erlangga\thanks{Zayed University, Department of Mathematics, Abu Dhabi Campus, P.O. Box 144534, United Arab Emirates (yogi.erlangga@zu.ac.ae)} \and Yerlan Amanbek\thanks{Nazarbayev University, Department of Mathematics, 53 Kabanbay Batyr Ave, Astana 010000, Kazakhstan (yerlan.amanbek@nu.edu.kz)}}
\maketitle


\begin{abstract}
    This study investigates tridiagonal near-Toeplitz matrices in which the Toeplitz part is strictly diagonally dominant. The focus is on determining the exact inverse of these matrices and establishing upper bounds for the infinite norms of the inverse matrices. For cases with $b > 2$ and $b < -2$, we derive the compact form of the entries of the exact inverse. These results remain valid even when the matrices' corners are not diagonally dominant, specifically when $|\wtb| < 1$.
     
    Furthermore, we calculate the traces and row sums of the inverse matrices. Afterwards, we present upper bound theorems for the infinite norms of the inverse matrices. To demonstrate the effectiveness of the bounds and their application, we provide numerical results for solving Fisher's problem. 
    Our findings reveal that the converging rates of fixed-point iterations closely align with the expected rates, and there is minimal disparity between the upper bounds and infinite norm of the inverse matrix. Specifically, this observation holds true when $b > 2$ with $\widetilde{b} \leq 1$ and $b < -2$ with $\widetilde{b} \geq -1$.
    For other cases, there is potential for further improvement in the obtained upper bounds.

    This study contributes to the field of numerical analysis of fixed-point iterations by improving the convergence rate of iterations and reducing the computing time of the inverse matrices. 
\end{abstract}





\newcommand{\bs}[1]{\boldsymbol{#1}}

\section{Introduction}

Banded near-Toeplitz matrices are a  special structured matrix, which appears in many applications such as  in interpolation problems, modelling using differential equations \cite{diele1998use, fischer1974fourier, pozrikidis2014introduction, smith1985numerical, yueh2008explicit}, as well as in time series analysis \cite{luati2010spectral, reddi1984eigenvector} and auto-regressive models \cite{akaike1973block, akansu2012toeplitz, tan2019explicit}. In solving differential equations, for instance, a finite difference discretization of the differential equation results in a system of equations which involved a banded near-Toeplitz matrix, which is perturbed slightly from the Toeplitz structure due to the inclusion of the boundary conditions. In nonlinear differential equations, a fixed-point iteration can be employed to solve the nonlinear system. In this case, one often times is interested in the study of the convergence of the method, which involves the inverse of the underlying near-Toeplitz matrix.

As an illustration, consider the differential equation 
\begin{equation}
	  \displaystyle \frac{d^2 u}{dx^2} = f(u),  \quad  x \in \Omega = (0,L). \notag 
\end{equation}
with some boundary conditions at $x \in \{0,L\}$ and some nonlinear function $u$. Examples are the so-called Fisher equation, with $f(u)=ku(1-u)$, and the Gelfand-Bratu equation, with $f(u)=ke^u$. Consider an approximation to the second derivative using the central difference scheme
 \begin{equation}
      \frac{d^2 u}{dx^2} (x_i) \approx \frac{1}{h^2} \left( u_{i-1} - 2u_{i}+u_{i+1}\right),
\end{equation}
where $x_i = ih \in [0,L]$, with $h$ be a constant step size and $i = 1,\dots,n = L/h$, and $u_i \equiv u(x_i)$. This approximation process results in a system of nonlinear equations, which can be written as 
\begin{equation}
    \widehat{T}_n\boldsymbol{u}=h^2f(\boldsymbol{u}) \label{eq:nonlinear}
\end{equation}
where $\widehat{T}_n$ is an $n$-by-$n$ tridiagonal near Toeplitz matrix and $\boldsymbol{u} = (u_1, \dots, u_n)^T \in \mathbb{R}^n$. Starting from an initial guess of the solution $\boldsymbol{u}^0$, an approximate solution of the differential equation can be computed via the fixed-point recurrence
\begin{equation}
    \boldsymbol{u}^{k+1}=h^2 \widehat{T}^{-1}_n f(\boldsymbol{u}^{k}), \quad k = 1,2,\dots,
\end{equation}
provided that the recurrence converges. Assuming that $f$ is a Lipschitz continuous function in the domain $\Omega$, the sequence of solutions generated by the fixed-point iteration satisfies the inequality, for $p \in \{1,2,\infty\}$,
\begin{align*}
    \| \boldsymbol{u}^{k+1}-\boldsymbol{u}^{k}\|_p =\|h^2 \widehat{T}^{-1}_n \left( f(\boldsymbol{u}^{k})-f(\boldsymbol{u}^{k-1})\right)\|_p \le h^2 \| \widehat{T}^{-1}_n  \|_p  L_c \| \boldsymbol{u}^{k}-\boldsymbol{u}^{k-1}\|_p,
\end{align*}
where $L_c$ is the Lipschitz constant, which depends on the nature of $f(\boldsymbol{u})$. To ensure the convergence of the fixed point iteration method, we require $h^2 L_c \| \widehat{T}^{-1}_n  \|_p < 1$. Here, an upper bound of $\| \widehat{T}^{-1}_n  \|_p$ is useful in the convergence analysis.


Inverses of some classes of Toeplitz matrices have been explored extensively  in, for example, ~\cite{dow2002explicit, heinig2013algebraic, huang1997analytical, lewis1982inversion, schlegel1970explicit, usmani1994inversion, wang2015explicit, yamamoto1979inversion}. Related to the inverse Toeplitz matrices are their inverse properties, bounds for norms of inverses~\cite{allgower1973exact, bottcher1999convergence, kurmanbek2021explicit, kurmanbek2021inverse, pan2015estimating, turkmen2002bounds}, and their determinants~\cite{amanbek2020explicit, andjelic2021some, kurmanbek2022proof, lv2008note, shitov2021determinants, sweet1969recursive}.

In this paper, we present inverse properties of $n\times n$ symmetric tridiagonal near-Toeplitz matrix
\begin{equation}
   \wtT_n = \begin{bmatrix}
             \overline{b} & \hat{c} &   & &\\
             \hat{c} & \hat{b} &  \ddots & &\\
               &  \ddots   &  \ddots & \ddots & \\
               &              &  \ddots   & \hat{b} & \hat{c} \\
               &              &  &  \hat{c}  & \overline{b}
            \end{bmatrix} = -\hat{c}
            \underbrace{ \begin{bmatrix}
              \widetilde{b} & -1 &   & &\\
             -1 & b &  \ddots & &\\
               &  \ddots   &  \ddots & \ddots & \\
               &              &  \ddots  & b & -1 \\
               &              &   & -1  & \widetilde{b} 
            \end{bmatrix}}_{\widetilde{T}_n} =: -\hat{c} \, \widetilde{T}_n, \label{eq:matrixTn}
\end{equation}
with $n \ge 3$, such as the one arising in~\eqref{eq:nonlinear}. Here, $b = -\hat{b}/\hat{c}$, and $\widetilde{b} = -\overline{b}/\hat{c}$. Since $\widehat{T}_n$ is a factor of $\widetilde{T}_n$, we will primarily focus on the latter. 

The near-Toeplitz matrix $\wtT_n$ in \eqref{eq:matrixTn} can be expressed as a rank-2 decomposition
\begin{equation}
    \widetilde{T}_n = T_n + (\widetilde{b}-b) UU^T, \label{eq:rank-2}
\end{equation}
where $T_n = \text{tridiag}(-1,b,-1)$ is the Toeplitz part of $\widetilde{T}_n$ and the matrix $U^T$ is given by
\begin{equation}
   U^T = \begin{bmatrix}
                  1 & 0 & \dots & 0 & 0 \\
                  0 & 0 & \dots & 0 & 1
              \end{bmatrix}.
\end{equation}
The inverse of the Toeplitz part $T_n$, if it exists, can be characterized and explicitly expressed by the roots of the polynomial $p(r) = -r^2 + br - 1$ (see \cite{dow2002explicit}). If $T_n$ is diagonally dominant, meaning $|b| \ge 2$, then $p(r)$ has either two distinct real roots or equal roots, with the latter occurring when $b = 2$. Otherwise, $p(r)$ has two complex roots.

Suppose the inverse of $\widetilde{T}_n$ exists. Applying the Sherman-Morrison formula, the inverse of $\widetilde{T}_n$ is given by
\begin{equation}
  \widetilde{T}^{-1}_n = T^{-1}_n - (\widetilde{b} - b) T^{-1}_n U M^{-1} U^T T^{-1}_n, \label{eq:invTtilde}
\end{equation}
where $M = I_2 + (\widetilde{b}-b) U^T T^{-1}_{n} U \in \mathbb{R}^{2\times 2}$. 
Let $\widetilde{T}^{-1}_n =: [\wtt^{-1}_{i,j}]$, $T^{-1}_n =: [t^{-1}_{i,j}]$, and $\beta = \wtb - b$. Then, for $M =: [m_{i,j}]_{i,j = 1,2}$, we have
\begin{align}
   m_{11} &= m_{22} =  1 + (\widetilde{b} - b) t^{-1}_{1,1} = 1 + \beta t^{-1}_{n, n},\notag \\
   m_{12} &= m_{21} = (\widetilde{b}-b) t^{-1}_{1,n} = \beta t^{-1}_{n, 1}, \notag
\end{align}
due to the symmetry and centrosymmetry of $T_n$. Furthermore,
\begin{equation}
   M^{-1} = \frac{1}{\Delta} \begin{bmatrix}
              m_{11} & -m_{12} \\
              -m_{12} & m_{11}
           \end{bmatrix},
\end{equation}
where $\Delta = m_{11}^2  - m_{12}^2 = (1 + \beta t^{-1}_{n,n} )^2 - \beta^2  (t^{-1}_{n,1})^2$. With this, the $(i,j)$-entry (with $i \geq j$) of the inverse of $\widetilde{T}_n$ is given explicitly by the formula
\begin{equation}
 \wtt_{i,j}^{-1} =t^{-1}_{i,j} - \frac{\beta}{\Delta} \left[ t_{i,1}^{-1}(m_{11} t^{-1}_{1,j} - m_{12} t^{-1}_{n,j}) + t^{-1}_{i,n}(-m_{12} t^{-1}_{1,j} + m_{11} t^{-1}_{n,j})\right].  \label{eq:wtt}
\end{equation}
The trace of $\widetilde{T}^{-1}_n$ can be calculated using the formula
\begin{equation}
  \text{Tr}(\widetilde{T}^{-1}_n) := \sum_{i=1}^n \wtt^{-1}_{i,i} = \sum_{i=1}^n t^{-1}_{i,i} - \frac{\beta}{\Delta} \left\{ m_{11} \sum_{i=1}^n \left( (t^{-1}_{i,1})^2 + (t^{-1}_{n,i})^2 \right) - 2 m_{12} \sum_{i=1}^n t^{-1}_{i,1} t^{-1}_{n,i} \right\}, \label{eq:traceTt}
\end{equation}
due to the symmetry of $T_n$. The sum of the elements in the $i$-th row of the inverse is given by
\begin{align}
   \sum_{j=1}^n \wtt^{-1}_{i,j} &= \sum_{j=1}^n t^{-1}_{i,j} - \frac{\beta}{\Delta}\left[ t_{i,1}^{-1}(m_{11}\sum_{j=1}^n  t^{-1}_{1,j} - m_{12} \sum_{j=1}^n  t^{-1}_{n,j}) + t^{-1}_{i,n}(-m_{12} \sum_{j=1}^n  t^{-1}_{1,j} + m_{11} \sum_{j=1}^n  t^{-1}_{n,j})\right] \notag \\
&= \sum_{j=1}^n t^{-1}_{i,j} - \frac{\beta}{\Delta} \left[ t_{i,1}^{-1}(m_{11}- m_{12}) \sum_{j=1}^n  t^{-1}_{n,j} + t^{-1}_{i,n}(-m_{12}  + m_{11} ) \sum_{j=1}^n  t^{-1}_{n,j} \right]  \notag \\
&= \sum_{j=1}^n t^{-1}_{i,j} - \frac{\beta}{m_{11} + m_{12}} (t^{-1}_{i,1} + t^{-1}_{i,n}) \sum_{j=1}^n t^{-1}_{n,j}, \label{eq:rowsumTt}
\end{align}
after making use of the centrosymmetry of $T_n$ and $T^{-1}_n$, with
$$
  m_{11} + m_{12} = 1 + \beta(t^{-1}_{n,1} + t^{-1}_{n,n}).
$$

We note that inverse properties of symmetric tridiagonal near-Toeplitz matrices are discussed in~\cite{tan2019explicit} in relation with their applications in autoregressive modelling. The paper particularly derives an explicit formula for the trace and bounds for the rowsums of the inverse, where $b \neq 2$. Our paper discusses new bounds for the rowsum of  the inverse matrices and extend the analysis to bounds for norms. As we aim at effective bounds, we shall consider several different cases, which are reported separately: (i) $|b| > 2$, (ii) $|b| = 2$,and (iii) $|b| < 2$. This paper is devoted to the first case, where $|b| > 2$, namely where the Toeplitz part is strictly diagonally dominant.


The paper is organized as follows. In Section~\ref{sec:b>2}, we will derive  explicit formulas for the trace and bounds for the rowsum and the norms of the inverse matrix when $b > 2$. Section~\ref{sec:b<-2} will delve in the cases where $b < -2$. Section~\ref{sec:numexp} presents numerical experiments demonstrating the effectiveness of the upper bounds and a convergence rate comparison to Fisher's problem. 

\section{Preliminary results}

The inverse of the Toeplitz matrix $T_n = \text{tridiag}(-1,b,-1)$, with $|b| > 2$, in~\eqref{eq:rank-2} can be characterized by the two roots of the polynomial $p(r) = -r^2 + br - 1$: $r_1 = \frac{1}{2}(b + \sqrt{b^2-4})$ and $r_2 = \frac{1}{2}(b - \sqrt{b^2 - 4})$; see~\cite{dow2002explicit}. Let $\gamma_k = r_1^k - r_2^k$ with $\gamma_0 = 0$ and $\gamma_1 = \sqrt{b^2-4}$. Its inverse is given by
\begin{align}
  t^{-1}_{i,j} = \begin{cases}
                               \displaystyle \frac{\gamma_j \gamma_{n+1-i}}{\sqrt{b^2 - 4} \gamma_{n+1}},& i \ge j, \\
                               t^{-1}_{j,i},& i < j.
                       \end{cases} \label{eq:invTij}
\end{align}

Now, we list a couple of intermediate results, pertaining to $r_1, \gamma_k$, and $b$.

\begin{lemma} \label{lem:gamm1}
Let $b > 2$. For any integer $k \ge 2$, the following holds true:
\begin{enumerate}
\item[(i)] $\gamma_k + \gamma_{k-2} = b \gamma_{k-1}$;
\item[(ii)] $\frac{b}{2} < r_1 \le \frac{\gamma_{k+1}}{\gamma_k} \le b$;
\item[(iii)] $\gamma_{n+1-i}\gamma_{i+1} - \gamma_{i}\gamma_{n-i} = \gamma_{n+1}\gamma_{1}$.
\end{enumerate}
\end{lemma}

\begin{proof}
Using $\gamma_k = r_1^k - r_2^k$,
\begin{equation*}
\gamma_k + \gamma_{k-2} = r_1^{k-1}(r_1 + r_1^{-1}) - r_2^{k-1}(r_2 + r_2^{-1}),
\end{equation*}
from which Part (i) is derived via a direct calculation. From Part (i), $\gamma_{k+1} = b \gamma_k - \gamma_{k-1} \le b \gamma_k$. Also,
\begin{align*}
  \frac{\gamma_{k+1}}{\gamma_k} &= \frac{r_1^{k+1} - r_2^{k+1}}{r_1^k - r_2^k} = r_1 + r_2 - \frac{r_1^k r_2 - r_2^k r_1}{r_1^k - r_2^k} = b - \frac{r_1^{k-1} - r_2^{k-1}}{r_1^k - r_2^k} = b -  \frac{r_2 - r_1(r_2/r_1)^k}{1 - (r_2/r_1)^k},
\end{align*}
after using the fact that $r_1 r_2 = 1$. Since $r_1 > r_2$, 
\begin{align*}
   \frac{\gamma_{k+1}}{\gamma_k} \ge b - \frac{r_2(1 - (r_2/r_1)^k)}{1 - (r_2/r_1)^k} = b - r_2 = r_1 > \frac{b}{2},
\end{align*}
completing Part (ii). For Part (iii), using $r_1 r_2 = 1$,
\begin{align*}
    \gamma_{n+1-i}\gamma_{i+1} - \gamma_{i}\gamma_{n-i} &= r^{n+2}_{1} - r^{n-2i}_{1} - r^{n-2i}_{2} + r^{n+2}_{2} - r^{n}_{1} + r^{n-2i}_{1} + r^{n-2i}_{2} - r^{n}_{2}\\
    &= r^{n+1}_{1} r^{1}_{1} - r^{n+1}_{1} r^{1}_{2} - r^{n+1}_{2} r^{1}_{1} + r^{n+1}_{2} r^{1}_{2} = \gamma_{n+1} \gamma_{1}.
\end{align*}
\end{proof}

\begin{lemma}  \label{lem:gamm2}
Let $|b| > 2$. For any integer $p \ge 1$, the following relations hold:
\begin{enumerate}
  \item[(i)] $\displaystyle \sum_{k=1}^p \gamma_k = \frac{1}{b-2}(\gamma_{p+1} - \gamma_p - \gamma_1)$;
  \item[(ii)] $\displaystyle \sum_{k=1}^p k \gamma_k = \frac{1}{b-2}(p \gamma_{p+1} - (p+1) \gamma_p)$.
\end{enumerate}
\end{lemma}

\begin{proof}
For Part (i),
\begin{align*}
  \sum_{k=1}^p \gamma_k 
    &= \sum_{k=1}^p (r_1^k - r_2^k) =  \sum_{k=0}^p r_1^k - 1 - \sum_{k=0}^p r_2^k + 1  = \frac{r_1^{p+1} - 1}{r_1 - 1} - \frac{r_2^{p+1} - 1}{r_2 - 1} \\
    &= \frac{r_1^{p+1} - r_2^{p+1} - (r_1^p - r_2^p) - (r_1 - r_2)}{b-2} = \frac{1}{b-2}(\gamma_{p+1} - \gamma_p - \gamma_1).
\end{align*}
For Part (ii),
\begin{align*}
  \sum_{k=1}^p k\gamma_k 
    &= \sum_{k=1}^{p} k(r_1^k - r_2^k) = r_1 \sum_{k=1}^p kr_1^{k-1} - r_2 \sum_{k=1}^p kr_2^{k-1} \\
    &= r_1 \left( \sum_{k=0}^p r_1^k \right)' - r_2 \left( \sum_{k=0}^p r_2^k\right)' = r_1 \left(  \frac{r_1^{p+1} -1}{r_1-1} \right)' - r_2 \left( \frac{r_2^{p+1} - 1}{r_2 -1}\right)' \\
    &= \frac{pr_1^{p+2} - (p+1)r_1^{p+1} + r_1}{(r_1-1)^2} - \frac{pr_2^{p+2} - (p+1)r_2^{p+1} + r_2}{(r_2 -1)^2} \\
    &= \frac{p r_1^{p+1} - (p+1)r_1^p + 1}{b-2} - \frac{p r_2^{p+1} - (p+1)r_2^p + 1}{b-2}\\
    &= \frac{1}{b-2}(p\gamma_{p+1} - (p+1)\gamma_{p}).
\end{align*}
\end{proof}

\begin{lemma} \label{lem:gamm3} 
For $b > 2$ and $n \ge 3$, 
$$b - 1 < b - \frac{2}{b} \leq \frac{\gamma_{n+1}}{\gamma_n + \gamma_1} < r_1 < b.$$
\end{lemma}
\noindent The detailed proof of Lemma~\ref{lem:gamm3} is provided in Appendix \ref{E}.  

\section{Trace, rowsum, and norms of the inverse matrix: $b > 2$} \label{sec:b>2}

In this section, we derive inverse properties of the near Toeplitz matrix $\wtT_n$, when $b > 2$. Our first result is on non-singularity of $\wtT_n$.

\begin{lemma}[Non-singularity of $\wtT_n$] \label{lem:invertTt}
Let $\wtT_n$ be as defined in~\eqref{eq:matrixTn} with $b > 2$. Then $\wtT_n$ is nonsingular if $\wtb \neq b - \frac{\gamma_{n+1}}{\gamma_n + \gamma_1} := \wtb_1$ or $\wtb \neq b - \frac{\gamma_{n+1}}{\gamma_n - \gamma_1} := \wtb_2$. Furthermore, it holds that $0 \leq \wtb_2 < \wtb_1 < 1$.
\end{lemma}

\begin{proof}
The first term on the right-hand side of~\eqref{eq:wtt} exists for $b > 2$, and nonsingularity requires that the second term involving $\Delta$ does not vanish. In this case, with $t^{-1}_{n,1} = \gamma_1/\gamma_{n+1}$ and $t^{-1}_{n,n} = \gamma_n/\gamma_{n+1}$, we have
\[
\Delta = \frac{\gamma_{n+1}^2 + 2(\wtb - b)\gamma_n \gamma_{n+1} + (\wtb - b)^2 (\gamma_n^2 - \gamma_1^2)}{\gamma_{n+1}^2} = 0
\]
for $\wtb \in \{\wtb_1,\wtb_2\}$ defined in the lemma. It is worth noting that $\wtb_1 - \wtb_2 = \frac{2\gamma_1 \gamma_{n+1}}{\gamma_n^2 - \gamma_1^2} > 0$, which implies that $\wtb_1 > \wtb_2$. Furthermore, with $b - 1 < \frac{\gamma_{n+1}}{\gamma_n + \gamma_1}$ (Lemma~\ref{lem:gamm3}), we have $b - \frac{\gamma_{n+1}}{\gamma_n + \gamma_1} = \wtb_1 < 1$. Finally, for non-negativity of $\wtb_2$,
\[
\wtb_2 = b - \frac{\gamma_{n+1}}{\gamma_n - \gamma_1} = \frac{b \gamma_n - b \gamma_1 - \gamma_{n+1}}{\gamma_n - \gamma_1} = \frac{\gamma_{n-1} - b \gamma_1}{\gamma_n - \gamma_1} = \frac{\gamma_{n-1} - \gamma_2}{\gamma_n - \gamma_1} \ge 0,
\]
after using Lemma~\ref{lem:gamm1}(i) and the fact that $\gamma_2 = b \gamma_1$, with equality occurring when $n = 3$.
\end{proof}

Next, we introduce the compact form of the element of $\widetilde{T}^{-1}_{n}$. This form is particularly important when we derive some upper bounds when considering the case $\widetilde{b} < 1$. The proof is presented in Appendix~\ref{F}.

\begin{lemma}\label{lem:compact_t}
Let $b > 2$ and $\widetilde{b}$ such that $\widetilde{T}_{n}$ is invertible. For $i \geq j$, we have 
\[
\widetilde{t}^{-1}_{i, j} = C (\gamma_{n+1-i} + \beta \gamma_{n-i})(\gamma_{j} + \beta \gamma_{j-1}),
\]
where $C = \displaystyle \frac{\gamma_{n+1}}{\gamma_1 ((\gamma_{n+1} + \beta \gamma_n)^2 - (\beta \gamma_1)^2)}$ and $\beta = \widetilde{b} - b$.
\end{lemma}

\subsection{Trace of the inverse}

The general formula for the trace of the inverse matrix $\wtT_n^{-1}$ is given by~\eqref{eq:traceTt}. With $t^{-1}_{n,1} = \gamma_1/\gamma_{n+1}$, $t^{-1}_{n,n} = \gamma_n/\gamma_{n+1}$, $m_{11} = 1 + (\wtb-b)\gamma_n/\gamma_{n+1}$,  and $m_{12} = (\wtb-b)\gamma_1/\gamma_{n+1}$, the diagonal entries of the inverse matrix $\wtT^{-1}_n$ are given by 
\begin{align}
   \wtt^{-1}_{i,i} &= t^{-1}_{i,i} - \frac{\wtb-b}{\Delta} \left[ \frac{\gamma_{n+1} + (\wtb-b)\gamma_n}{\gamma_{n+1}}\left( \left( \frac{ \gamma_{n+1-i}}{ \gamma_{n+1}} \right)^2 + \left( \frac{ \gamma_i}{ \gamma_{n+1}} \right)^2\right) -  2(\wtb-b)\frac{\gamma_1}{\gamma_{n+1}} \frac{\gamma_i \gamma_{n+1-i}}{\gamma_{n+1}^2} \right] \notag \\
    &= \left[ 1 + \frac{2(\wtb-b)^2 (b^2-4)}{\Delta \gamma_{n+1}^2} \right] t^{-1}_{i,i} - \frac{(\wtb-b)(\gamma_{n+1}+(\wtb-b)\gamma_n)}{\Delta \gamma^3_{n+1}} (\gamma_i^2 + \gamma_{n+1-i}^2).
\end{align}
The trace formula~\eqref{eq:traceTt} can then be written as
\begin{align}
   \text{Tr}(\wtT^{-1}_n) = \left[ 1 + \frac{2(\wtb-b)^2 (b^2-4)}{\Delta \gamma_{n+1}^2} \right] \sum_{i = 1}^n t^{-1}_{i,i} - \frac{(\wtb-b)(\gamma_{n+1}+(\wtb-b)\gamma_n)}{\Delta \gamma^3_{n+1}} \sum_{i=1}^n \left\{ \gamma_i^2 + \gamma_{n+1-i}^2 \right\}.
\end{align}
The first sum in the right-hand side is the trace of $T^{-1}_n$, whose formula is given by the following lemma:
\begin{lemma} \label{lem:traceTn}
For $b > 2$ in $T_n$, 
\begin{align}
   \text{Tr}(T^{-1}_n) = \frac{n+1}{\sqrt{b^2-4}} \left( 1 + \frac{2 r_2^{n+1}}{\gamma_{n+1}} \right) - \frac{b}{b^2-4}.
\end{align}
\end{lemma} 
\begin{proof}
Expressing the trace of the inverse matrix \( T_{n}^{-1} \) in terms of its elements given by~\eqref{eq:invTij}, we have
    \begin{align*}
        \mathrm{Tr}(T_{n}^{-1}) &= \sum_{i=1}^{n} t^{-1}_{i, i} = \frac{1}{\gamma_{1}\gamma_{n+1}} \sum_{i=1}^{n} \left( r^{n+1}_{1} + r^{n+1}_{2} - r^{n+1-2i}_{1} - r^{n+1-2i}_{2}\right)\\
        &= \frac{n+1}{\gamma_1\gamma_{n+1}} (r^{n+1}_{1} + r^{n+1}_{2}) - \frac{1}{\gamma_{1}\gamma_{n+1}} \sum_{i=0}^{n} \left(r^{n+1-2i}_{1} + r^{n+1-2i}_{2}\right)\\
        &= \frac{n+1}{\gamma_1\gamma_{n+1}} (r^{n+1}_{1} + r^{n+1}_{2}) - \frac{1}{\gamma_{1}\gamma_{n+1}} \left[\frac{r^{n-1}_{1} + r^{n-1}_{2} - r^{n+3}_{1} - r^{n+3}_{2}}{(1 - r^{2}_{1})(1 - r^{2}_2)}\right]\\
        &= \frac{n+1}{\gamma_1\gamma_{n+1}} (r^{n+1}_{1} + r^{n+1}_{2}) + \frac{1}{\gamma_{1}\gamma_{n+1}} \frac{\gamma_{n+1} \gamma_2}{4 - b^2} \quad \text{[by $\gamma_2 = b \gamma_1$]}\\
        &= \frac{n+1}{\sqrt{b^2 - 4}} \left(1 + \frac{2r^{n+1}_2}{\gamma_{n+1}}\right) - \frac{b}{b^2 - 4}.
    \end{align*}
\end{proof}

\noindent For the second sum, we have the following lemma, which is proved in Appendix~\ref{G}.
\begin{lemma} \label{lem:sum_squared_gammas}
For $b > 2$,
\begin{align}
   \sum_{i=1}^n \left\{ \gamma_i^2 + \gamma_{n+1-i}^2 \right\} = \frac{b\gamma_{n+1}}{\sqrt{b^2-4}} \left( \gamma_{\frac{n+1}{2}}^2 + 2 \right) - \gamma_{n+1}^2 - 4(n+1). \notag
\end{align}
\end{lemma}

The following theorem gives an explicit formula for $\text{Tr}(\wtT^{-1}_n)$, which is proved using Lemmas~\ref{lem:traceTn} and ~\ref{lem:sum_squared_gammas}:

\begin{theorem}[Trace]
Let $\wtb$ be such that, for $b > 2$, $\wtT_n$ is nonsingular. Then 
$$
  \text{Tr}(\wtT^{-1}_n) = K_1 \left( \frac{n+1}{\sqrt{b^2-4}} \frac{\gamma^2_{\frac{n+1}{2}} + 2}{\gamma_{n+1}} - \frac{b}{b^2-4} \right) - K_2 \left( \frac{b}{\sqrt{b^2-4}} \left( \gamma^2_{\frac{n+1}{2}} + 2\right) - \gamma_{n+1}- \frac{4(n+1)}{\gamma_{n+1}} \right),
$$
where $K_1 = \displaystyle \frac{[\gamma_{n+1} + (\wtb-b)\gamma_n]^2 + (\wtb-b)^2 (b^2-4)}{[\gamma_{n+1} + (\wtb-b)\gamma_n]^2 - (\wtb-b)^2 (b^2-4)}$ and $K_2 = \displaystyle  \frac{(\wtb-b)[\gamma_{n+1} + (\wtb-b)\gamma_n]}{[\gamma_{n+1} + (\wtb-b)\gamma_n]^2 - (\wtb-b)^2 (b^2-4)}$.
\end{theorem}

\subsection{Rowsums of the inverse}

The general formula for the rowsum of the inverse of $\wtT_n$ is given by \eqref{eq:rowsumTt}.  We need to evaluate each term in the right-hand side of Equation \eqref{eq:rowsumTt}.

For the first term of the right-hand side of~\eqref{eq:rowsumTt}, with $t^{-1}_{i,j}$ given in~\eqref{eq:invTij},
\begin{align}
  \sum_{j=1}^n t^{-1}_{i,j} &= \sum_{j=1}^n t^{-1}_{i,j} + \sum_{j=i+1}^n t^{-1}_{i,j} = \sum_{j=1}^n t^{-1}_{i,j} + \sum_{j=i+1}^n t^{-1}_{j,i} 
  \notag \\
   &= \frac{\gamma_{n+1-i}}{\sqrt{b^2-4}} \sum_{j=1}^i \gamma_j + \frac{\gamma_i}{\sqrt{b^2-4}\gamma_{n+1}} \sum_{k=1}^{n-i} \gamma_k \notag \\
  &= \frac{\gamma_{n+1-i}}{\sqrt{b^2-4} \gamma_{n+1}} \left(\frac{1}{b-2}(\gamma_{i+1} - \gamma_i - \gamma_1  \right) + \frac{\gamma_i}{\sqrt{b^2 - 4} \gamma_{n+1}} \left(  \frac{1}{b-2}(\gamma_{n+1-i} - \gamma_{n-i} - \gamma_1 \right) \notag \\
  &= \frac{1}{(b-2)\sqrt{b^2-4} \gamma_{n+1}} \left( \gamma_{n+1-i}\gamma_{i+1}  - \gamma_1 \gamma_{n+1-i} - \gamma_i \gamma_{n-i} - \gamma_1 \gamma_i \right) \notag \\
  &= \frac{\gamma_{n+1} - \gamma_{n+1-i} - \gamma_{i}}{(b-2)\gamma_{n+1}}, \label{12}
\end{align}
after the use of Lemma \ref{lem:gamm2} and Lemma \ref{lem:gamm1}(iii).

For the second term of the right-hand side of~\eqref{eq:rowsumTt}, with $t^{-1}_{n,1} = \gamma_1/\gamma_{n+1}$ and $t^{-1}_{n,n} = \gamma_n/\gamma_{n+1}$, we have
\begin{equation}\label{13}
  m_{11} + m_{12} = 1 + (\wtb - b) \frac{\gamma_1 + \gamma_n}{\gamma_{n+1}}.
\end{equation}
Next, with $t^{-1}_{i,1} = \gamma_{n+1-i}/\gamma_{n+1}$ and $t^{-1}_{i,n} = t^{-1}_{n,i} = \gamma_i/\gamma_{n+1}$, we  get
\begin{equation} \label{14}
   t^{-1}_{i,1} +t^{-1}_{i,n} = \frac{\gamma_{n+1-i} + \gamma_i}{\gamma_{n+1}}.
\end{equation}
Finally,
\begin{equation}\label{15}
\sum_{j=1}^n t^{-1}_{n,j} = \sum_{j=1}^n \frac{\gamma_j \gamma_1}{\sqrt{b^2-4} \gamma_{n+1}} = \frac{1}{\gamma_{n+1}} \sum_{j=1}^n \gamma_j = \frac{1}{(b-2)\gamma_{n+1}} \left( \gamma_{n+1} - \gamma_n - \gamma_1 \right).
\end{equation}
due to Lemma~\ref{lem:gamm2}.
Substitution of \eqref{12} -- \eqref{15} to Equation \eqref{eq:rowsumTt} gives
\begin{align} 
   R(i) :&= \sum_{j=1}^n \wtt^{-1}_{i.j} = \underbrace{\frac{\gamma_{n+1} - \gamma_{n+1-i} - \gamma_{i}}{(b-2)\gamma_{n+1}}}_{R_1(i)}  +    \underbrace{\frac{(b - \wtb)(\gamma_{n+1-i} + \gamma_i)\left( \gamma_{n+1} - \gamma_n - \gamma_1 \right)}{(b-2)\gamma_{n+1}(\gamma_{n+1} + (\wtb - b)(\gamma_1 + \gamma_n))}}_{R_2(i)} \notag \\
   &= \frac{1}{b-2} + \frac{b-\wtb-1}{b-2} \frac{\gamma_i + \gamma_{n+1-i}}{\gamma_{n+1} + (\wtb - b)(\gamma_n + \gamma_1)}.  \label{eq:rowsum1}
\end{align}



We begin with the positivity of the inverse matrix $\wtT_n$ when $\wtb \geq 1$, summarized in the following lemma.

\begin{lemma}[Positivity]
Let  $b > 2$, $\wtb \geq 1$, and $n \ge 3$. The inverse matrix $\wtT_n^{-1}$ is positive.
\end{lemma}

\begin{proof}
Note that, by Lemma~\ref{lem:invertTt}, the matrix $\wtT_n$ is nonsingular. Consider the decomposition $\wtT_n = S_n + \wtD_n$, where $S_n$ is defined by $\wtT_{n}$ in \eqref{eq:matrixTn} with $b = 2$, $\wtb = 1$, and $\wtD_n = \text{diag}(\wtb - 1, b-2,\dots,b-2,\wtb-1)$. For any $x \in \mathbb{R}^n$, we have
\[
x^T \wtT_n x = 
x^TS_n x + x^T \wtD_n x = x^T S_n x + (\wtb-1)x_1^2 + (b-2)\sum_{k=2}^{n-1} x_k^2 + (\wtb-1)x_n^2 > 0,
\]
because $S_n$ is symmetric positive semidefinite, $\wtb - 1 > 0$, and $b-2 > 0$. Therefore, $\wtT_n$ is a Stieltjes matrix, whose inverse is positive.
\end{proof}

\begin{lemma} \label{lem:rowsum1}
Let $b > 2$ and $\wtb \ge 1$. For the rowsums of $\wtT^{-1}_n$, the following holds:
\begin{enumerate}
\item[(i)] if $\wtb > b-1$, then the maximum of rowsums is attained at $i = (n+1)/2$ and the minimum is attained at $i = 1$ and $i = n$;
\item[(ii)]if $\wtb = b - 1$, then the rowsums are constant;
\item[(iii)] if $b - \frac{\gamma_{n+1}}{\gamma_n + \gamma_1} < \wtb < b-1$, then the maximum of rowsums is attained at $i = 1$ and $i = n$ and the minimum is attained at $i = (n+1)/2$; and
\item[(iv)] if $\wtb < b - \frac{\gamma_{n+1}}{\gamma_n + \gamma_1}$, then the maximum of rowsums is attained at $i = (n+1)/2$ and the minimum is attained at $i = 1$ and $i = n$;
\end{enumerate}
\end{lemma}
\begin{proof}
Note that $\wtT_n$ and, hence, $\wtT^{-1}_n$ are centrosymmetric. Thus the rowsum $R(i)$ is symmetric along $i = (n+1)/2$. Differentiating $R(i)$ gives $R'(i) = R'_1(i) + R'_2(i)$, where
\begin{align}
   R'_1(i) &= \frac{1}{(b-2)\gamma_{n+1}}\left(( r_1^{n+1-i} - r^i_1) \ln r_1 - (r^{n+1-i}_2 - r^i_2) \ln r_2\right), \notag \\
   R'_2(i) &= \frac{(\wtb - b)(\gamma_{n+1}-\gamma_n - \gamma_1)}{(b-2)\gamma_{n+1}(\gamma_{n+1} + (\wtb-b)(\gamma_1 + \gamma_n))} \left((r^{n+1-i}_1 - r_1^i) \ln r_1 -( r^{n+1-i}_2 - r^i_2) \ln r_2\right). \notag
\end{align}
Thus,
$$
  R'(i) = \frac{1 + \wtb - b}{(b-2) (\gamma_{n+1} + (\wtb-b)(\gamma_1 + \gamma_n))}\left((r^{n+1-i}_1 - r^i_1)\ln r_1 - (r_2^{n+1-i} - r_2^i )\ln r_2\right.
$$
The above derivative vanishes if (i) $\wtb = b - 1$, $\forall i$, implying that $R(i)$ is constant, or (ii) $i = (n+1)/2$, the critical point of $R(i)$.
The type of the critical point can be evaluated using the second derivative of $R$:
$$
  R''(i) = - \frac{1 + \wtb - b}{(b-2) (\gamma_{n+1} + (\wtb-b)(\gamma_1 + \gamma_n))}\left[\ln^2 r_2 \left( (r_1^{n+1-i} + r^i) \frac{\ln^2 r_1}{\ln^2 r_2} - (r_2^{n+1-i}+ r_2^i) \right) \right].
$$
The bracketed term is always positive. So the sign of $R''((n+1)/2)$ depends on $\wtb$, which can be determined via standard algebra and leads to the results stated in the lemma. 
\end{proof}

\begin{theorem}\label{lem:rowsum2}
Let $\wtT_n$ be as in~\eqref{eq:matrixTn}, with $b > 2$ and $\wtb \ge 1$. Then the rowsums of the inverse matrix $\wtT^{-1}_n$ satisfy the following:
\begin{enumerate}
   \item[(i)] for $\wtb > b-1$,  
       $$
           \frac{1}{\wtb-1} < R(i) < \frac{1}{b-2} \left( 1 - \frac{2(1 + \wtb - b)(b-1)}{\wtb} \frac{\gamma_{\frac{n+1}{2}}}{\gamma_{n+1}} \right);
       $$
   \item[(ii)] for $\wtb = b-1$, 
        $$R(i) = \frac{1}{b-2};$$
   \item[(iii)] for $\frac{2}{b} <  \wtb < b-1$, 
        $$
            \frac{1}{b-2} \left( 1 + \frac{2(b - \wtb -1) (b-1)}{\wtb} \frac{\gamma_{\frac{n+1}{2}}}{\gamma_{n+1}} \right) < R(i) < \frac{b+1}{\wtb b - 2};
        $$
   \item[(iv)] for $\wtb < r_2$, 
  $$
\frac{r_1 - 1}{(b-2)(\wtb - r_2)}< R(i) < \frac{1}{b-2} \left( 1 +  \frac{2(b  - \wtb - 1)(b-1)}{\wtb-1} \frac{\gamma_{\frac{n+1}{2}}}{\gamma_{n+1}}\right).
  $$
\end{enumerate}
\end{theorem}

\begin{proof}
From the formula of the rowsum~\eqref{eq:rowsum1}, 
\begin{align}
   R(1) &= \frac{\gamma_{n+1} - \gamma_n - \gamma_1}{(b-2)(\gamma_{n+1} + (\wtb-b)(\gamma_n + \gamma_1))}, \label{eq:R1} \\
   R\left(\frac{n+1}{2}\right) &= \frac{1}{b-2} - \frac{2(1 + \wtb - b)}{b-2} \frac{\gamma_{\frac{n+1}{2}}}{\gamma_{n+1} + (\wtb-b)(\gamma_n + \gamma_1)}. \label{eq:Rn12}
\end{align}
Setting $\wtb = b - 1$, we have $R(1) = 1/(b-2)$. By Lemma~\ref{lem:rowsum1}(ii), $R(i) = R(1)$, proving Part (ii).

For Part (i), with $\wtb > b -1$, 
\begin{align}
  R(1) = \frac{1}{(b-2)\left( 1 + \displaystyle \frac{(1 + \wtb -b) (\gamma_n + \gamma_1)}{\gamma_{n+1} - \gamma_n - \gamma_ 1} \right)} >  \frac{1}{(b-2)\left( 1 + \frac{1+\wtb -b}{b-2} \right)} = \frac{1}{\wtb-1}, \notag
\end{align}
after using Lemma~\ref{lem:gamm3}. For~\eqref{eq:Rn12}, note that
\begin{align*}
  1 + \frac{(\wtb-b)(\gamma_n + \gamma_1)}{\gamma_{n+1}} &= 1 - \frac{\gamma_n + \gamma_1}{\gamma_{n+1}} + \frac{(1 + \wtb -b)(\gamma_n + \gamma_1)}{\gamma_{n+1}} \\
&< 1 - \frac{1}{b} + \frac{(1+\wtb-b)}{b-1} = -\frac{1}{b} + \frac{\wtb}{b-1} < \frac{\wtb}{b-1}.
\end{align*}
Therefore,
\begin{align*}
  R\left(\frac{n+1}{2}\right) &= \frac{1}{b-2} -  \frac{2(1 + \wtb - b)}{b-2} \frac{1}{1+\frac{(\wtb-b)(\gamma_n + \gamma_1)}{\gamma_{n+1}}} \frac{\gamma_{\frac{n+1}{2}}}{\gamma_{n+1}} \\
&< \frac{1}{b-2} -  \frac{2(1 + \wtb - b)}{b-2} \frac{1}{\wtb/(b-1)} \frac{\gamma_{\frac{n+1}{2}}}{\gamma_{n+1}} = \frac{1}{b-2} \left( 1 - \frac{2(1+\wtb-b)(b-1)}{\wtb} \frac{\gamma_{\frac{n+1}{2}}}{\gamma_{n+1}}\right).
\end{align*}
Combining the two inequalities and Lemma~\ref{lem:rowsum1} leads to Part (i).

For Part (iii), due to Lemma~\ref{lem:rowsum1}, with $1 + \wtb - b < 0$. Using inequalities we have, we can see that 
\begin{align*}
    1 + \frac{(1+\wtb - b)(\gamma_n + \gamma_1)}{\gamma_{n+1} - \gamma_n - \gamma_1} > 1 + \frac{1+\wtb-b}{b-1-\frac{2}{b}} = \frac{\wtb - \frac{2}{b}}{b-1-\frac{2}{b}}
\end{align*}
And since $ 1 + \frac{(1+\wtb - b)(\gamma_n + \gamma_1)}{\gamma_{n+1} - \gamma_n - \gamma_1} > 0$, for $\frac{2}{b} < \wtb < b - 1$, we get the following upper bound
\begin{align*}
    R(i) \le R(1) = \frac{1}{b-2} \frac{1}{1 + \frac{(1 + \wtb -b)(\gamma_n + \gamma_1)}{\gamma_{n+1}-\gamma_n-\gamma_1}} < \frac{b - 1 - \frac{2}{b}}{(b - 2)(\wtb - \frac{2}{b})} = \frac{b^2 - b - 2}{(b-2)(\wtb b - 2)} = \frac{b+1}{\wtb b - 2}
\end{align*}

For the lower bound, 
\begin{align*}
  R(i) \ge R(\frac{n+1}{2}) &= \frac{1}{b-2} - \frac{2(1+\wtb -b)}{b-2} \frac{\gamma_{\frac{n+1}{2}}}{\gamma_{n+1} + (\wtb - b)(\gamma_n + \gamma_1)} \\
&= \frac{1}{b-2} + \frac{2(b-\wtb-1)}{b-2} 
\frac{1}{1 + \frac{(1+ \wtb -b)(\gamma_n + \gamma_1)}{\gamma_{n+1} - \gamma_n - \gamma_1}} \frac{\gamma_{\frac{n+1}{2}}}{\gamma_{n+1} - \gamma_n - \gamma_1} \\
  &> \frac{1}{b-2} + \frac{2(b-\wtb-1)}{b-2} 
\frac{1}{1 + (1+ \wtb -b)\frac{1}{b-1}} \frac{\gamma_{\frac{n+1}{2}}}{\gamma_{n+1} - \gamma_n - \gamma_1} \\
&= \frac{1}{b-2} + \frac{2(b-\wtb-1)(b-1)}{(b-2)\wtb} 
 \frac{\gamma_{\frac{n+1}{2}}}{\gamma_{n+1}}  \frac{\gamma_{n+1}}{\gamma_n + \gamma_1} \frac{\gamma_n + \gamma_1}{\gamma_{n+1} - \gamma_n - \gamma_1} \\
&> \frac{1}{b-2} \left( 1 + \frac{2(b-\wtb-1)(b-1)}{\wtb} \frac{\gamma_{\frac{n+1}{2}}}{\gamma_{n+1}} \right),
\end{align*}
after applying Lemma~\ref{lem:gamm3}. 

For Part (iv), first note that
\begin{align*}
    \frac{\gamma_{n+1}}{\gamma_n + \gamma_1} + \wtb - b < r_1 + \wtb - b = \wtb - r_2
\end{align*}
Since $\gamma_{n+1} + (\wtb - b)(\gamma_n + \gamma_1) < 0$, for all $\wtb < r_2$, we have 
\begin{align*}
    R(1) = \frac{1}{b-2} \frac{\gamma_{n+1} - \gamma_n - \gamma_1}{\gamma_n + \gamma_1} \frac{1}{\frac{\gamma_{n+1}}{\gamma_n + \gamma_1} + \wtb - b} &> \frac{1}{(b-2) (\wtb - r_2)} \frac{\gamma_{n+1} - \gamma_n - \gamma_1}{\gamma_n + \gamma_1} > \frac{r_1 - 1}{(b-2) (\wtb - r_2)}.
\end{align*}
Furthermore,
\begin{align*}
    1 + \frac{(\wtb - b)(\gamma_n + \gamma_1)}{\gamma_{n+1}} > 1 + \frac{\wtb - b}{b-1}  = \frac{\wtb - 1}{b - 1}.
\end{align*}
Therefore,
\begin{align*}
    R(i) \leq R\left(\frac{n+1}{2}\right) &= \frac{1}{b-2} + \frac{2(b-1-\wtb)}{b-2} \cdot \frac{\gamma_{\frac{n+1}{2}}}{\gamma_{n+1}} \frac{1}{1+ \frac{(\wtb - b)(\gamma_{n}+\gamma_{1})}{\gamma_{n+1}}}\\ &< \frac{1}{b-2} + \frac{2(b-1-\wtb)(b-1)}{(b-2)(\wtb - 1)} \frac{\gamma_{\frac{n+1}{2}}}{\gamma_{n+1}}.
\end{align*}
\end{proof}

For $b - \frac{\gamma_{n+1}}{\gamma_n + \gamma_1} < \wtb \leq \frac{2}{b}$, we however cannot obtain a meaningful upper bound; the derived upper bound grows faster towards infinity and does not capture the growth of the exact trace of the inverse towards infinity as $\wtb$ approaches either end of this interval. Similarly, we at the moment are not able to derived a meaningful lower bound for the trace of the inverse when $r_2 \leq \wtb < b - \frac{\gamma_{n+1}}{\gamma_{n} + \gamma_1}$.

\subsection{Norms of the inverse matrix}

Since $\wtT^{-1}_n \ge 0$ for $\wtb \ge 1$, we have $|\wtt^{-1}_{i,j}| = \wtt^{-1}_{i,j}$. Hence,  $\sum_{j=1}^n|\wtt^{-1}_{i,j}| = \sum_{j=1}^n \wtt^{-1}_{i,j}$, and the $\infty$-norm of the inverse matrix can be bounded from above using the same upper bound as for the rowsum. Thus,

\begin{theorem}
Let $b > 2$ and $\wtb \ge 1$ such that $\wtT_n$ is nonsingular. Then
\begin{align*}\label{upper_bound_1}
   \|\wtT^{-1}_n\|_{\infty} \le \begin{cases} 
                               \displaystyle   \frac{1}{b-2} \left( 1 - \frac{2(1+\wtb-b)(b-1)}{\wtb}\frac{\gamma_{\frac{n+1}{2}}}{\gamma_{n+1}} \right), &\text{for } \wtb > b - 1,  \\
                                            \displaystyle \frac{1}{b-2},  &\text{for } \wtb = b - 1,\\
                                  \displaystyle \frac{b+1}{\wtb b - 2}, 
                                  &\text{for } 
                                  1 \le \wtb < b-1. 
                                            \end{cases}
\end{align*}
\end{theorem}

For $\wtb < 1$, an upper bound for the $\infty$-norm of $\wtT_n^{-1}$ is given in the following theorem:

\begin{theorem}\label{theor_upper_bound_2}
Let $\wtb < 1$ be such that, for $b > 2$, the matrix $\wtT_{n}$ is nonsingular. Then, the following inequality holds 
\begin{align}
    \|\wtT^{-1}_{n}\|_{\infty} \leq 
    \frac{\gamma_{n+1} - 2\gamma_{\frac{n+1}{2}}}{(b-2)\gamma_{n+1}} + \frac{|K|}{b-2}(\gamma_{n} + \gamma_{1})(\gamma_{n+1}- \gamma_{n} -\gamma_{1}) + \frac{|K \beta|}{b-2}\gamma_{n-1}\gamma_{n+1},
\end{align}
where $K = \frac{\beta}{(\gamma_{n+1} + \beta \gamma_{n})^2 - (\beta \gamma_1)^2}$ and $\beta = \wtb - b$.
\end{theorem}
\begin{proof}
We can rewrite \eqref{eq2_wtt_init} in Appendix \ref{F} as follows for $i \geq j$
\begin{align*}
    \wtt^{-1}_{i, j} = \frac{\gamma_j \gamma_{n+1-i}}{\gamma_1 \gamma_{n+1}} - K (\gamma_{n+1-i} \gamma_{n+1-j} + \gamma_i \gamma_j) - K\beta (\gamma_{n+1-i} \gamma_{n-j} + \gamma_{i} \gamma_{j-1}),
\end{align*}
where $K = \frac{\beta}{(\gamma_{n+1} + \beta \gamma_{n})^2 - (\beta \gamma_1)^2}$. Therefore, we can bound its absolute value as
\begin{align*}
    \left| \wtt^{-1}_{i, j} \right| \leq \frac{\gamma_j \gamma_{n+1-i}}{\gamma_1 \gamma_{n+1}} + |K| (\gamma_{n+1-i} \gamma_{n+1-j} + \gamma_i \gamma_j) + |K\beta| (\gamma_{n+1-i} \gamma_{n-j} + \gamma_{i} \gamma_{j-1}).
\end{align*}
For $j \geq i+1$, we have
\begin{align*}
    \left| \wtt^{-1}_{i, j}\right| = \left| \wtt^{-1}_{j, i}\right| \leq \frac{\gamma_i \gamma_{n+1-j}}{\gamma_1 \gamma_{n+1}} + |K| (\gamma_{n+1-i} \gamma_{n+1-j} + \gamma_i \gamma_j) + |K\beta| (\gamma_{n+1-j} \gamma_{n-i} + \gamma_{j} \gamma_{i-1}).
\end{align*}
First note that
\begin{align*}
   \sum_{j=1}^{n}\{\gamma_{n+1-i} \gamma_{n+1-j} + \gamma_{i} \gamma_{j}\} &= \gamma_{n+1-i} \sum_{j=1}^{n} \gamma_{n+1-j} + \gamma_{i} \sum_{j=1}^{n} \gamma_{j}
    = (\gamma_{n+1-i} + \gamma_{i}) \sum_{j=1}^{n} \gamma_{j} \\
    &= \frac{1}{b-2}(\gamma_{n+1-i} + \gamma_{i}) (\gamma_{n+1} -\gamma_{n} - \gamma_{1}),
\end{align*}
and
\begin{align*}
   \sum_{j=1}^n \{\gamma_{n+1-j} \gamma_{n-i} + \gamma_{j} \gamma_{i-1} \} &= \sum_{j=1}^{i}\{\gamma_{n+1-i} \gamma_{n-j} + \gamma_{i} \gamma_{j-1}\} + \sum_{j=i+1}^{n} \{\gamma_{n+1-j}\gamma_{n-i} + \gamma_{j} \gamma_{i-1}\}\\
    &= \gamma_{n+1-i} \sum_{j=1}^{i} \gamma_{n-j} + \gamma_{i-1}\sum_{j=1}^{i} \gamma_{j-1} + \gamma_{n-i} \sum_{j=i+1}^{n} \gamma_{n+1-j} + \gamma_{i-1}\sum_{j=i+1}^{n}\gamma_{j}\\
    &= \frac{\gamma_{n+1-i}}{b-2}(\gamma_{n} - \gamma_{n-1} - \gamma_{n-i} + \gamma_{n-i-1}) + \frac{\gamma_{i-1}}{b-2} (\gamma_i - \gamma_{i-1} -\gamma_1) \\
    &+ \frac{\gamma_{n-i}}{b-2}(\gamma_{n+1-i} - \gamma_{n-i} - \gamma_1) + \frac{\gamma_{i-1}}{b-2}(\gamma_{n+1} - \gamma_{n} - \gamma_{i+1} + \gamma_{i}) \\
    &= \frac{1}{b-2}(\gamma_{n+1}\gamma_{n-i} - \gamma_{n-1}\gamma_{n+1-i} + \gamma_{n+1}\gamma_{i-1}).
\end{align*}

Therefore,
\begin{align}
    \sum_{j=1}^{n} \left| \wtt^{-1}_{i, j} \right| &\leq \sum_{j=1}^{n} t^{-1}_{i, j} + |K| \sum_{j=1}^n \{ \gamma_{n+1-i} \gamma_{n+1-j} + \gamma_i \gamma_j  \}  + |K\beta| \sum_{j=1}^n \{\gamma_{n+1-i} \gamma_{n-j} + \gamma_{i} \gamma_{j-1}\}, \notag \\
    &\leq \frac{\gamma_{n+1} -\gamma_{n+1-i} -\gamma_{i}}{(b-2)\gamma_{n+1}} + \frac{|K|}{b-2} (\gamma_{n+1-i} + \gamma_i) (\gamma_{n+1} - \gamma_{n} - \gamma_{1}) \notag \\ &+ \frac{|K \beta|}{b-2} (\gamma_{n+1}\gamma_{n-i} - \gamma_{n-1}\gamma_{n+1-i} + \gamma_{n+1}\gamma_{i-1}). \label{eq:normble1}
\end{align}

Now, note that for the first term on the right-hand side, with $\ln r_2 = - \ln r_1$,
\begin{align*}
    (\gamma_{n+1} - \gamma_{n+1-i} - \gamma_{i})' &= (r^{n+1-i}_{1} - r^{i}_{1})\ln{r_1} + (r^{i}_{2} - r^{n+1-i}_{2})\ln{r_2} \\
    &=  (r^{n+1-i}_{1} + r^{n+1-i}_{2} - r^{i}_{1}  - r^{i}_{2})\ln{r_1}\\
    &=  \gamma_{\frac{n+1-2i}{2}}\gamma_{\frac{n+1}{2}} \ln{r_1}.
\end{align*}
The critical point corresponds to $\gamma_{\frac{n+1-2i}{2}} = 0$, which is attained at $i = \frac{n+1}{2}$ (i.e., $\gamma_0 = 0$). Thus,
\begin{align*}
    \max_{i} \{\gamma_{n+1} - \gamma_{n+1-i} - \gamma_{i}\}  = \gamma_{n+1} - 2 \gamma_{\frac{n+1}{2}}.
\end{align*}
Applying a similar analysis to $\gamma_{n+1-i} + \gamma_i$ in the second term of the right-hand side, we obtain
$$
\max_{i} \{\gamma_{n+1-i} + \gamma_{i}\} \leq \gamma_{n} + \gamma_{1}.
$$
For the last term, note that
$
\gamma_{n+1}\gamma_{n-i} - \gamma_{n-1}\gamma_{n+1-i} + \gamma_{n+1}\gamma_{i-1} < \gamma_{n+1} (\gamma_{n-i} + \gamma_{i-1})
$, where
\begin{align*}
    \max_{i} \{\gamma_{n-i} + \gamma_{i-1}\} \leq \gamma_{n-1} + \gamma_{0} = \gamma_{n-1}.
\end{align*}
Substituting the above bounds for the maximum to the inequality~\eqref{eq:normble1} leads to the bound stated in the theorem.

\end{proof}
To derive a sharper upper bound, consider the case where $i \geq j$. Using \eqref{eq2_wtt}, we express $\left| \wtt^{-1}_{i, j} \right|$ as follows
\begin{align*}
    \left| \wtt^{-1}_{i, j} \right| = \frac{\left| \gamma_j + \beta \gamma_{j-1}\right|}{\gamma_1} \left|\wtt^{-1}_{i, 1}\right|.
\end{align*}
For the case when $j \geq i+1$, the expression becomes
\begin{align*}
    \left| \wtt^{-1}_{i, j} \right| = \frac{\left| \gamma_{n+1-j} + \beta \gamma_{n-j} \right|}{\gamma_1} \left| \wtt^{-1}_{n, i} \right|.
\end{align*}
Combining these cases, the overall summation is represented by Equation \eqref{eq2_wsum}
\begin{align} \label{eq2_wsum}
    \sum_{j=1}^{n} \left|\wtt^{-1}_{i, j} \right| 
    &= \sum_{j=1}^{i} \left|\wtt^{-1}_{i, j} \right| + \sum_{j=i+1}^{n} \left|\wtt^{-1}_{i, j} \right|\notag\\
    &= \frac{\left| \wtt^{-1}_{i, 1} \right|}{\gamma_1} \sum_{j=1}^{i}\left|\gamma_{j} + \beta \gamma_{j-1} \right| + \frac{\left|\wtt^{-1}_{n, i} \right|}{\gamma_1} \sum_{j=i+1}^{n} \left|\gamma_{n+1-j} + \beta \gamma_{n-j} \right|.
\end{align}
This representation allows us to study the sign change of $\gamma_{j} + \beta \gamma_{j-1}$ and $\gamma_{n+1-j} + \beta \gamma_{n-j}$, providing a more accurate upper bound.

\section{Trace, rowsum, and norms of the inverse matrix: $b < -2$} \label{sec:b<-2}

In this case, let $b_{+} = - b$ and $\wtb_{+} = -\wtb$, where $b < -2$, so $b_{+} > 2$. We make the following observations:
\begin{align*}
    r_{1} &= \frac{1}{2}(b + \sqrt{b^2 - 4}) = - \frac{1}{2}(b_{+} - \sqrt{b^{2}_{+} - 4}) =: - r_{2, +}, \\
    r_{2} &= \frac{1}{2}(b - \sqrt{b^2 - 4}) = - \frac{1}{2}(b_{+} + \sqrt{b^{2}_{+} - 4}) =: - r_{1, +}.    
\end{align*}
Furthermore,
\begin{align*}
    \gamma_{k} &= r^{k}_{1} - r^{k}_{2} = (-1)^{k+1}(r^{k}_{1, +} - r^{k}_{2, +}) = (-1)^{k+1} \gamma_{k, +}, \\
    \beta &= \widetilde{b} - b = -\widetilde{b}_{+} + b_{+} = - \beta_{+}.
\end{align*}
Using the above relations, entries of $T_n^{-1}$ with $b<-2$ can be related to the $b > 2$ case in the following way: for $i \geq j$
\begin{align*}
t^{-1}_{i,j} &= \frac{\gamma_{j} \gamma_{n+1-i}}{\gamma_{1}\gamma_{n+1}} = \frac{(-1)^{j+1} \gamma_{j, +} (-1)^{n+2-i}\gamma_{n+1-i, +}}{(-1)^2 \gamma_{1, +} (-1)^{n+2}\gamma_{n+1, +}} = (-1)^{j+1-i} \frac{\gamma_{j, +} \gamma_{n+1-i, +}}{\gamma_{1, +} \gamma_{n+1, +}} = (-1)^{i-j-1} t^{-1}_{i,j,+},
\end{align*}
where $t_{i,j,+}^{-1}$ is as in \eqref{eq:invTij}.

Now, denote the index of diagonals of $T_{n}$ by $s_d = i - j$, where $s_d = 0$ and $1$, e.g., indicate respectively the main diagonal and one diagonal below the main diagonal, and so on. If $s_d$ is odd, then $t^{-1}_{i, j} = t^{-1}_{i, j, +}$, and if $s_d$ is even, then $t^{-1}_{i, j} = - t^{-1}_{i, j, +}$. Utilizing \eqref{eq2_wtt} in Appendix~\ref{F}, for $i \geq j$
\begin{align*}
    \widetilde{t}^{-1}_{i, j} &= C (\gamma_{n+1-i} + \beta \gamma_{n-i})(\gamma_{j} + \beta \gamma_{j - 1}) \\
    &= C \left[(-1)^{n+2-i} \gamma_{n+1-i, +} - (-1)^{n+1-i}\beta_{+}\gamma_{n-i}) ((-1)^{j+1}\gamma_{j, +} - (-1)^{j}\beta_{+}\gamma_{j-1, +}\right] \\
    &= C (-1)^{n+1-i+j} (\gamma_{n+1-i, +} + \beta_{+} \gamma_{n-i, +}) (\gamma_{j, +} + \beta_{+} \gamma_{j-1, +}),
\end{align*}
where $C = \frac{\gamma_{n+1}}{\gamma_{1} ((\gamma_{n+1} + \beta \gamma_{n})^2 - (\beta \gamma_{1})^2)} = (-1)^{n} \frac{\gamma_{n+1, +}}{\gamma_{1, +} ((\gamma_{n+1, +} + \beta_{+} \gamma_{n, +})^2 - (\beta_{+} \gamma_{1, +})^2)}$. Using this, we have
\begin{align*}
    \widetilde{t}^{-1}_{i, j} = (-1)^{i - j - 1} \widetilde{t}^{-1}_{i, j, +}, \quad i \geq j,
\end{align*}
where $\widetilde{t}^{-1}_{i, j, +}$ is the $(i, j)$-th element of the $\widetilde{T}^{-1}_{n, +}$ matrix in \eqref{eq:matrixTn} with $b = b_{+}$ and $\wtb = -\wtb_{+}$.

\begin{lemma}[Nonsingularity]
Let $\wtT_{n}$ be as defined in \eqref{eq:matrixTn}, with $b < -2$. Then $\wtT_{n}$ is nonsingular if $\wtb = b - \frac{\gamma_{n+1}}{\gamma_{n} + \gamma_{1}} := \wtb_{1}$ or $\wtb = b - \frac{\gamma_{n+1}}{\gamma_{n} - \gamma_{1}} := \wtb_{2}$.
\end{lemma}
\begin{proof}
$t^{-1}_{i,j}$ in~\eqref{eq:invTij} is always defined for $b < -2$. So, the existence is $\wtt^{-1}_{i,j}$ given by \eqref{eq:wtt}  requires that  $\Delta \neq 0$. 
    \begin{align*}
        \Delta &= (1 + \beta t^{-1}_{n, n})^2 - (\beta t^{-1}_{n, 1})^2 \\
        &= \left(1 + \beta \left( \frac{\gamma_n + \gamma_1}{\gamma_{n+1}}\right) \right) \left( 1 + \beta \left(\frac{\gamma_n - \gamma_1}{\gamma_{n+1}} \right)\right) \neq 0
    \end{align*}
    From here, the result follows. 
\end{proof}
\begin{theorem}[Trace]
    Let $\wtb$ be such that, for $b < -2$, $\wtT_{n}$ is nonsingular. Then, the trace of the inverse of $\wtT_{n}$ is given by
    \begin{align*}
        \text{Tr}(\wtT^{-1}_{n}) &= K_1 \left( \frac{n+1}{\sqrt{b^2-4}} \frac{\gamma^2_{\frac{n+1}{2}} + 2}{\gamma_{n+1}} - \frac{b}{b^2-4} \right) \\
        &\quad - K_2 \left( \frac{b}{\sqrt{b^2-4}} \left( \gamma^2_{\frac{n+1}{2}} + 2\right) - \gamma_{n+1}- \frac{4(n+1)}{\gamma_{n+1}} \right)
    \end{align*}
    where $K_1 = \frac{[\gamma_{n+1} + \beta \gamma_{n}]^2 + \beta^2 (b^2 - 4)}{[\gamma_{n+1} + \beta \gamma_{n}]^2 - \beta^2 (b^2 - 4)}$ and $K_2 = \frac{\beta [\gamma_{n+1} + \beta \gamma_{n}]}{[\gamma_{n+1} + \beta \gamma_{n}]^2 - \beta^2 (b^2 - 4)}$ with $\beta = \wtb - b$. 
\end{theorem}

\begin{proof}
    As we know for $i \geq j$, we have $\wtt^{-1}_{i, j} = (-1)^{i-j-1} \wtt^{-1}_{i, j, +}$. This implies
    \begin{align*}
        \text{Tr}(\wtT^{-1}_{n}) = - \text{Tr}(\wtT^{-1}_{n, +})
    \end{align*}
    where $\wtT^{-1}_{n, +}$ with $b = b_{+} > 2$ and $\wtb  = -\wtb_{+}$. Using the trace formula from the case $b > 2$, we have 
    \begin{align*}
        \text{Tr}(\wtT^{-1}_{n, +}) &= K_{1, +} \left( \frac{n+1}{\sqrt{b^2_{+} - 4}} \frac{\gamma^2_{\frac{n+1}{2}, +} + 2}{\gamma_{n+1, +}} - \frac{b_{+}}{b^2_{+} - 4}\right) \\
        &\quad - K_{2,+} \left( \frac{b_{+}(\gamma^2_{\frac{n+1}{2}, +} + 2)}{\sqrt{b^2_{+} - 4}} - \gamma_{n+1, +} - \frac{4(n+1)}{\gamma_{n+1, +}}\right)
    \end{align*}
    where $K_{1, +} = \frac{[\gamma_{n+1,+} + \beta_{+} \gamma_{n, +}]^2 + \beta^2_{+} (b^2_{+} - 4)}{[\gamma_{n+1,+} + \beta_{+} \gamma_{n, +}]^2 - \beta^2_{+} (b^2_{+} - 4)}$ and $K_{2, +} = \frac{\beta_{+} [\gamma_{n+1, +} + \beta_{+} \gamma_{n, +}]}{[\gamma_{n+1,+} + \beta_{+} \gamma_{n, +}]^2 - \beta^2_{+} (b^2_{+} - 4)}$. 

    Now we want to express this formula in terms of $b, \beta$, and $\gamma$'s. Note
    \begin{align*}
        \frac{\gamma^2_{\frac{n+1}{2}, +} + 2}{\gamma_{n+1, +}} = - \frac{r^{n+1}_{1, +} + r^{n+1}_{2, +}}{r^{n+1}_{1, +} - r^{n+1}_{2, +}} = - \frac{r^{n+1}_{1} + r^{n+1}_{2}}{r^{n+1}_{1} - r^{n+1}_{2}} = - \frac{\gamma^{2}_{\frac{n+1}{2}} + 2}{\gamma_{n+1}}
    \end{align*}
    and 
    \begin{align*}
        \gamma^{2}_{\frac{n+1}{2}, +} + 2 &= r^{n+1}_{1, +} + r^{n+1}_{2, +} = (-1)^{n+1} \left(r^{n+1}_{1} + r^{n+1}_{2}\right) = (-1)^{n+1} \left( \gamma^{2}_{\frac{n+1}{2}} + 2 \right).
    \end{align*}
    Therefore, 
    \begin{align*}
        \frac{n+1}{\sqrt{b^2_{+} - 4}} \frac{\gamma^2_{\frac{n+1}{2}, +} + 2}{\gamma_{n+1, +}} - \frac{b_{+}}{b^2_{+} - 4} &= - \left(\frac{n+1}{\sqrt{b^2 - 4}} \frac{\gamma^2_{\frac{n+1}{2}} + 2}{\gamma_{n+1}} - \frac{b}{b^2 - 4} \right),\\
         \frac{b_{+}(\gamma^2_{\frac{n+1}{2}, +} + 2)}{\sqrt{b^2_{+} - 4}} - \gamma_{n+1, +} - \frac{4(n+1)}{\gamma_{n+1, +}} &= (-1)^{n+2} \left( \frac{b(\gamma^2_{\frac{n+1}{2}} + 2)}{\sqrt{b^2 - 4}} - \gamma_{n+1} - \frac{4(n+1)}{\gamma_{n+1}} \right).
    \end{align*}
    In the same way, 
    \begin{align*}
        K_{1, +} &= \frac{[\gamma_{n+1} + \beta \gamma_{n}]^2 + \beta^2 (b^2 - 4)}{[\gamma_{n+1} + \beta \gamma_{n}]^2 - \beta^2 (b^2 - 4)} = K_1 \\
        K_{2, +} &= \frac{(-1)^{n+3} \beta [\gamma_{n+1} + \beta \gamma_{n}]}{[\gamma_{n+1} + \beta \gamma_{n}]^2 - \beta^2 (b^2 - 4)} = (-1)^{n+3} K_{2}
    \end{align*}
   Combining all of this leads to the lemma.
\end{proof}
Since $r_1 r_2 = 1$ and $r_1 + r_2 = b$, we conclude that the rowsum of $\wtT^{-1}_{n}$ would be the same as in Equation~\eqref{eq:rowsum1}, so 
\begin{theorem}[Row Sums]
    Let $b < -2$, and $\wtb$ such that the matrix $\wtT_{n}$ is nonsingular. Then 
    \begin{equation}
        \text{rowsum}_{i} \wtT^{-1}_{n} = \frac{1}{b-2} - \frac{\beta + 1}{b-2} \frac{\gamma_{n+1-i} + \gamma_{i}}{\gamma_{n+1} + \beta (\gamma_1 + \gamma_n)}
    \end{equation}
    where $\beta =  \wtb - b$. 
\end{theorem}
\begin{theorem}[Upper Bound 1]\label{theor_upper_bound_3}
    Let $\wtb \leq -1$ be such that, for $b < -2$, the matrix $\wtT_{n}$ is nonsingular. Then, the following inequality holds
    \begin{align*}
        \|\wtT^{-1}_{n} \|_{\infty} \leq \begin{cases}
            \frac{1}{b + 2} \left(\frac{2(1-\wtb+b)(b+1)}{\wtb} \frac{\gamma_{\frac{n+1}{2}}}{\gamma_{n+1}}  -  1\right), &\text{if } \wtb < b + 1,\\
            -\frac{1}{b+2}, &\text{if } \wtb = b + 1,\\
            \frac{1 - b}{\wtb b - 2}, &\text{if } b + 1 < \wtb \leq -1.
        \end{cases}
    \end{align*}
\end{theorem}

\begin{proof}
    We observe that 
    \begin{align*}
        \| \wtT^{-1}_{n} \|_{\infty} &= \max_{i} \sum_{j=1}^{n} \left|\wtt^{-1}_{i, j}\right| = \max_{i} \sum_{j = 1}^{n} \left|\wtt^{-1}_{i, j, +}\right| = \|\wtT^{-1}_{n, +}\|_{\infty}.
    \end{align*}
    Using the relation
    \begin{align*}
        \frac{\gamma_{\frac{n+1}{2}, +}}{\gamma_{n+1, +}} &= \frac{r^{\frac{n+1}{2}}_{1, +} - r^{\frac{n+1}{2}}_{2, +}}{(-1)^{n+2}\gamma_{n+1}} \\
        &= \begin{cases}
            \frac{r^{\frac{n+1}{2}}_2 - r^{\frac{n+1}{2}}_1}{(-1)^{n+2}\gamma_{n+1}}, &\text{if } n \text{ is odd}, \\
            \frac{r^{\frac{n+1}{2}}_1 - r^{\frac{n+1}{2}}_2}{(-1)^{n+2}\gamma_{n+1}}, &\text{if } n \text{ is even},
        \end{cases}\\
        &= (-1)^{n}\frac{r^{\frac{n+1}{2}}_1 - r^{\frac{n+1}{2}}_2}{(-1)^{n+2}\gamma_{n+1}} \\
        &= \frac{\gamma_{\frac{n+1}{2}}}{\gamma_{n+1}},
    \end{align*}
    we conclude the proof of the theorem.
\end{proof}
\begin{theorem}[Upper Bound 2]\label{theor_upper_bound_4}
    Let $\wtb > -1$ be such that, for $b < -2$, the matrix $\wtT_{n}$ is nonsingular. Then, the following inequality holds
    \begin{align*}
        \|\wtT^{-1}_{n}\|_{\infty} &\leq - \frac{\gamma_{n+1} - 2\gamma_{\frac{n+1}{2}}}{(b+2)\gamma_{n+1}}  - \frac{|K|}{b+2} [(-1)^{n+1}\gamma_{n} + \gamma_1][(-1)^{n} (\gamma_{n+1} + \gamma_{n}) - \gamma_1]  - \frac{|K\beta|}{b + 2} \gamma_{n-1}\gamma_{n+1},
    \end{align*}
    where $K = -\frac{\beta}{(\gamma_{n+1} + \beta \gamma_{n})^2 - (\beta \gamma_1)^2}$.
\end{theorem}

\begin{proof}
    We begin by noting that $\|\wtT^{-1}_{n}\|_{\infty} = \|\wtT^{-1}_{n, +}\|_{\infty}$.
    Then, applying the inequality
    \begin{align*}
        \|\wtT^{-1}_{n, +}\|_{\infty} &\leq \frac{\gamma_{n+1, +} - 2\gamma_{\frac{n+1}{2}, +}}{(b_{+} - 2)\gamma_{n+1, +}}  + \frac{|K_{+}|}{b_{+}-2}(\gamma_{n, +} + \gamma_{1, +})(\gamma_{n+1, +} - \gamma_{n, +} - \gamma_{1, +}) \\
        &\quad + \frac{|K_{+}\beta_{+}|}{b_{+}-2}\gamma_{n-1, +}\gamma_{n+1, +},
    \end{align*}
    where $K_{+} = \frac{\beta_{+}}{(\gamma_{n+1, +} + \beta_{+} \gamma_{n, +})^2 - (\beta_{+} \gamma_{1, +})^2}$, we make the substitutions
    \begin{align*}
        \beta_{+} &= -\beta, \\
        \gamma_{k, +} &= (-1)^{k+1}\gamma_{k} \quad \text{for any $k \in \mathbb{Z}^{+}$}, \\
        b_{+} &= - b,
    \end{align*}
    to obtain the desired result.
\end{proof}

\section{Numerical Experiments} \label{sec:numexp}

In this section, we present an analysis of upper bounds and compare the maximum observed convergence rates for the fixed-point iteration method applied to Fisher's problem. Specifically, we note consistently tight upper bounds in cases where $b > 2$ with $\widetilde{b} \geq 1$ and when $b < -2$ with $\widetilde{b} \leq -1$. In other cases, there exists potential for improvement in alternate upper bounds.
\begin{figure}[htbp]
    \centering
    \begin{subfigure}{0.45\textwidth}
        \centering
        \includegraphics[width=\textwidth]{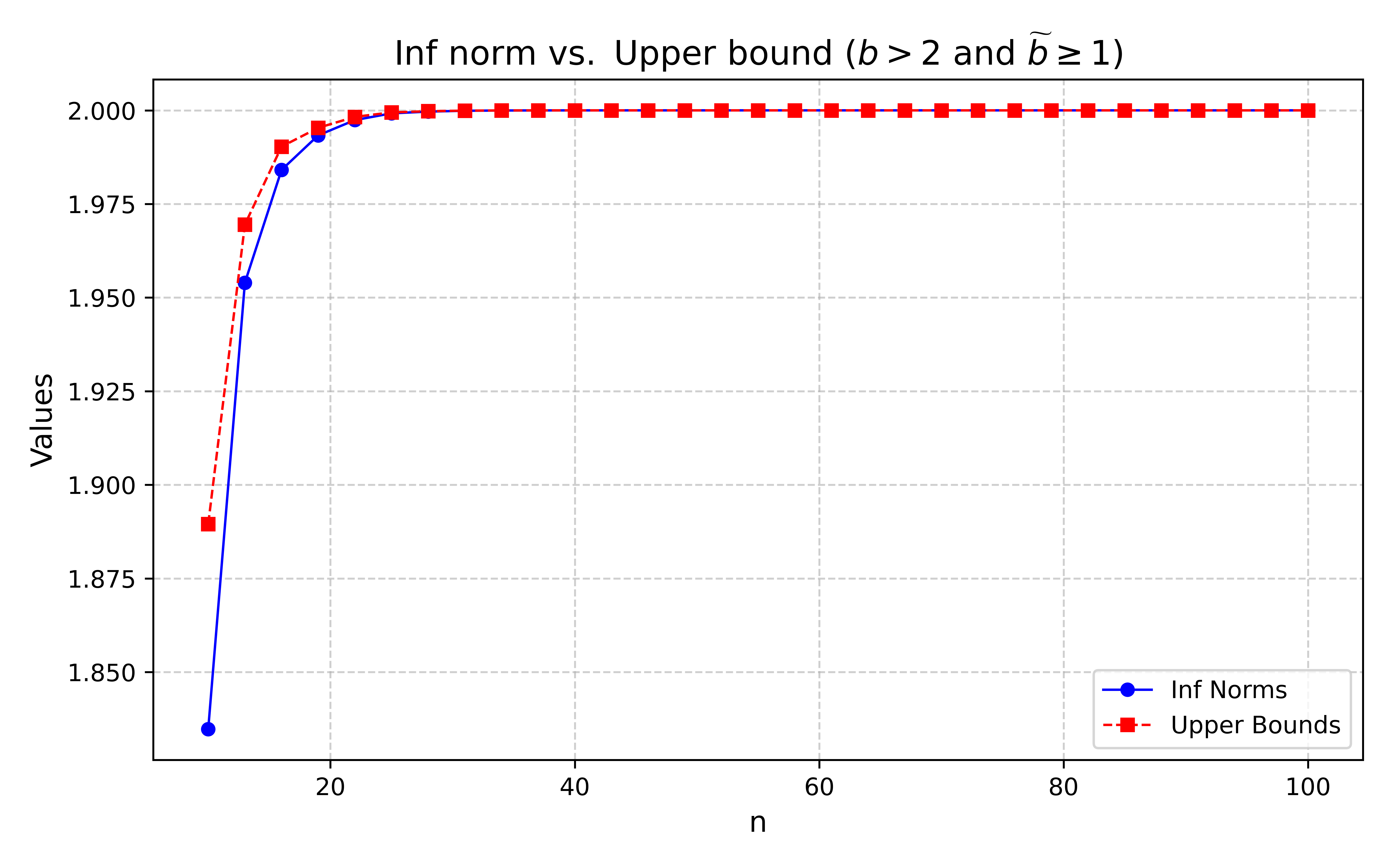}
        \caption{$\|\wtT^{-1}_{n}\|_{\infty}$ and upper bounds.}
        \label{fig3}
    \end{subfigure}
    \hspace{\fill} 
    \begin{subfigure}{0.45\textwidth}
        \centering
        \includegraphics[width=\textwidth]{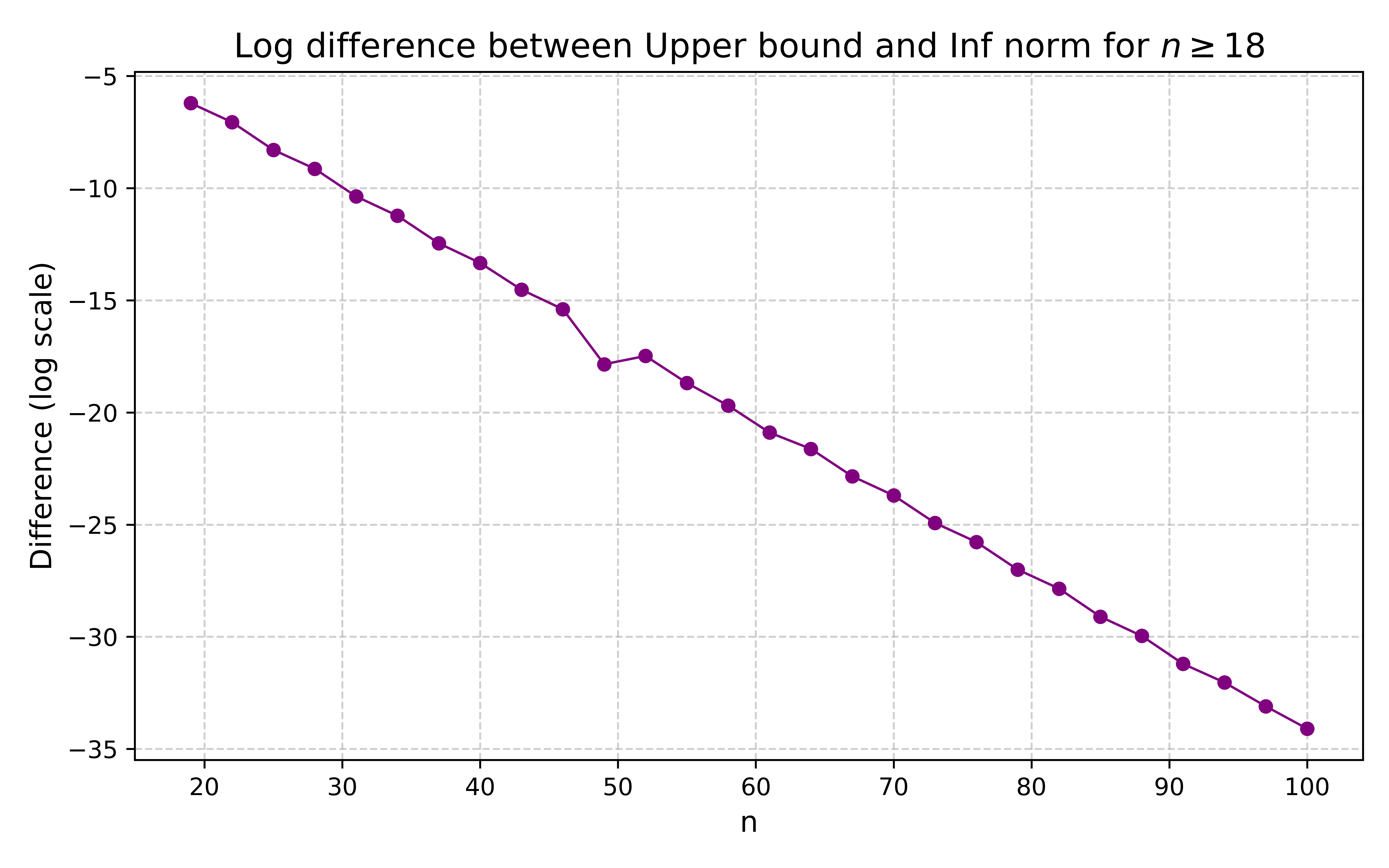}
        \caption{Log difference comparison.}
        \label{fig31}
    \end{subfigure}
    \caption{Comparison of $\|\wtT^{-1}_{n}\|_{\infty}$ vs. Upper bounds when $b > 2$ and $\wtb \geq 1$.}
    \label{fig:combined2}
\end{figure}

\begin{figure}[H]
    \centering
    \begin{subfigure}[b]{0.4\textwidth}
        \centering
       
        \label{tab2}
        \begin{tabular}{cccc}
            \toprule
            $n$ & $\widetilde{b}$ & $\|\wtT^{-1}_{n}\|_{\infty}$ & $\text{upper\_bound}$ \\
            \midrule
            10 &  0.09 & 5.532 & 154.994 \\
            13 & -1.14 & 1.062 & 11.274 \\
            16 & -1.32 & 0.999 & 9.859 \\
            19 & -0.35 & 2.222 & 32.339 \\
            22 &  0.16 & 7.238 & 254.557 \\
            25 & -0.73 & 1.454 & 17.118 \\
            28 & 0.94  & 2.888 & 25.091 \\
            \bottomrule
        \end{tabular}
         \caption{Comparison of $\|\wtT^{-1}_{n}\|_{\infty}$ and upper bounds.}
    \end{subfigure}%
    \hfill
    \begin{subfigure}[b]{0.45\textwidth}
        \centering
        \includegraphics[width=\textwidth]{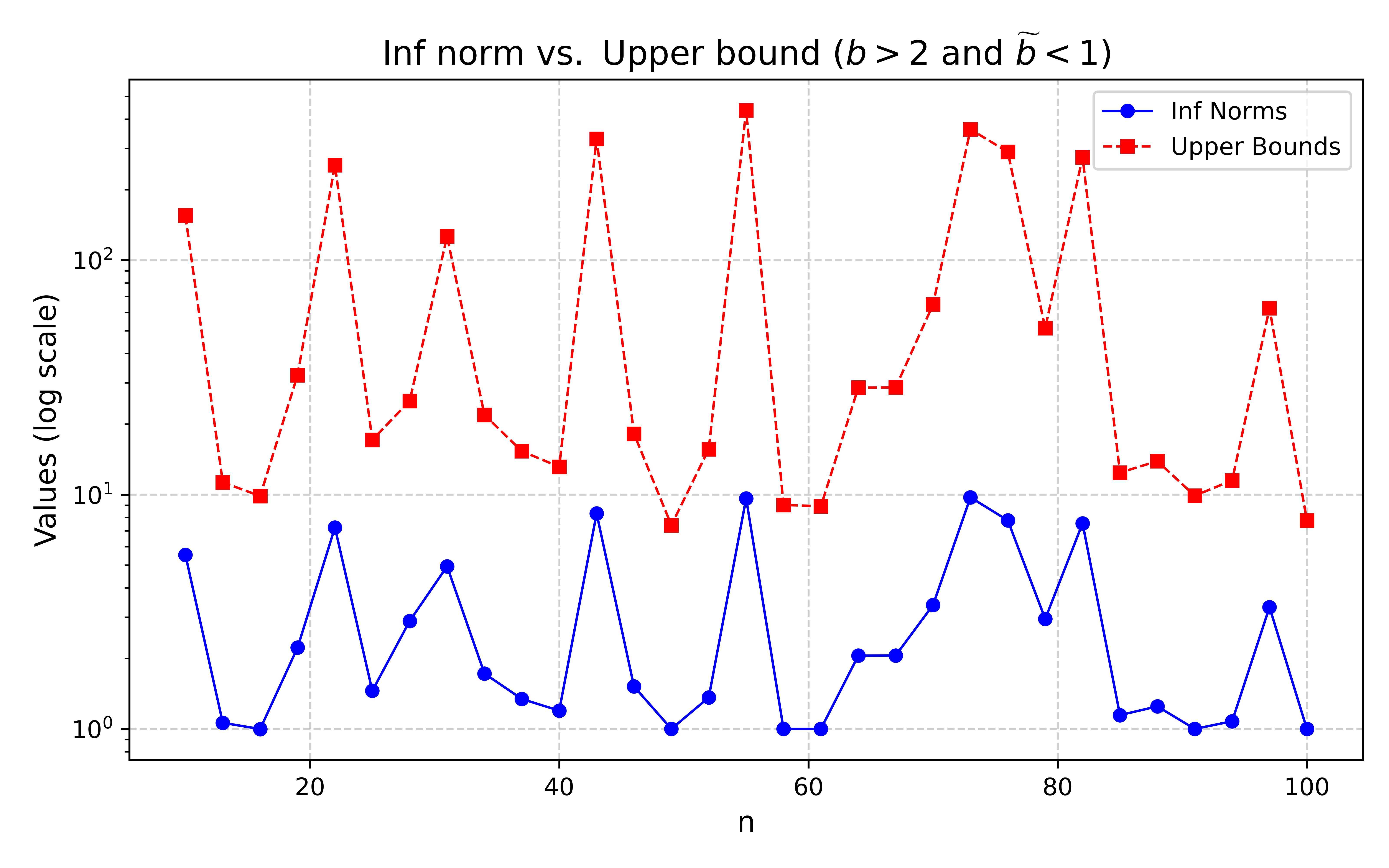}
        \caption{Log comparison of values.}
        \label{fig4}
    \end{subfigure}
    \caption{Comparison of $\|\wtT^{-1}_{n}\|_{\infty}$ vs. Upper bounds when $b>2$ and $\wtb < 1$.}
\end{figure}

\begin{figure}[htbp]
    \centering
    \begin{subfigure}{0.4\textwidth}
        \centering
        \label{tab3}
        \begin{tabular}{cccc}
            \toprule
            $n$ & $\widetilde{b}$ & $\|\wtT^{-1}_{n}\|_{\infty}$ & $\text{upper\_bound}$ \\
            \midrule
            10 &  -2.52  & 0.993 &  0.996 \\
            13 &  -4.57 & 0.996 &  0.997 \\
            16 &  -4.87  & 0.999 &  0.999 \\
            19 &  -3.24  &  0.999 &  0.999 \\
            22 &  -2.40  &  1.000 &  1.000 \\ 
            25 &  -3.88  &  1.000 &  1.000 \\
            28 &  -1.10  &  2.266 &  3.104 \\
            \bottomrule
        \end{tabular}
         \caption{Comparison of $\|\wtT^{-1}_{n}\|_{\infty}$ and upper bounds.}
    \end{subfigure}%
    \hfill
    \begin{subfigure}{0.45\textwidth}
        \centering
        \includegraphics[width=\textwidth]{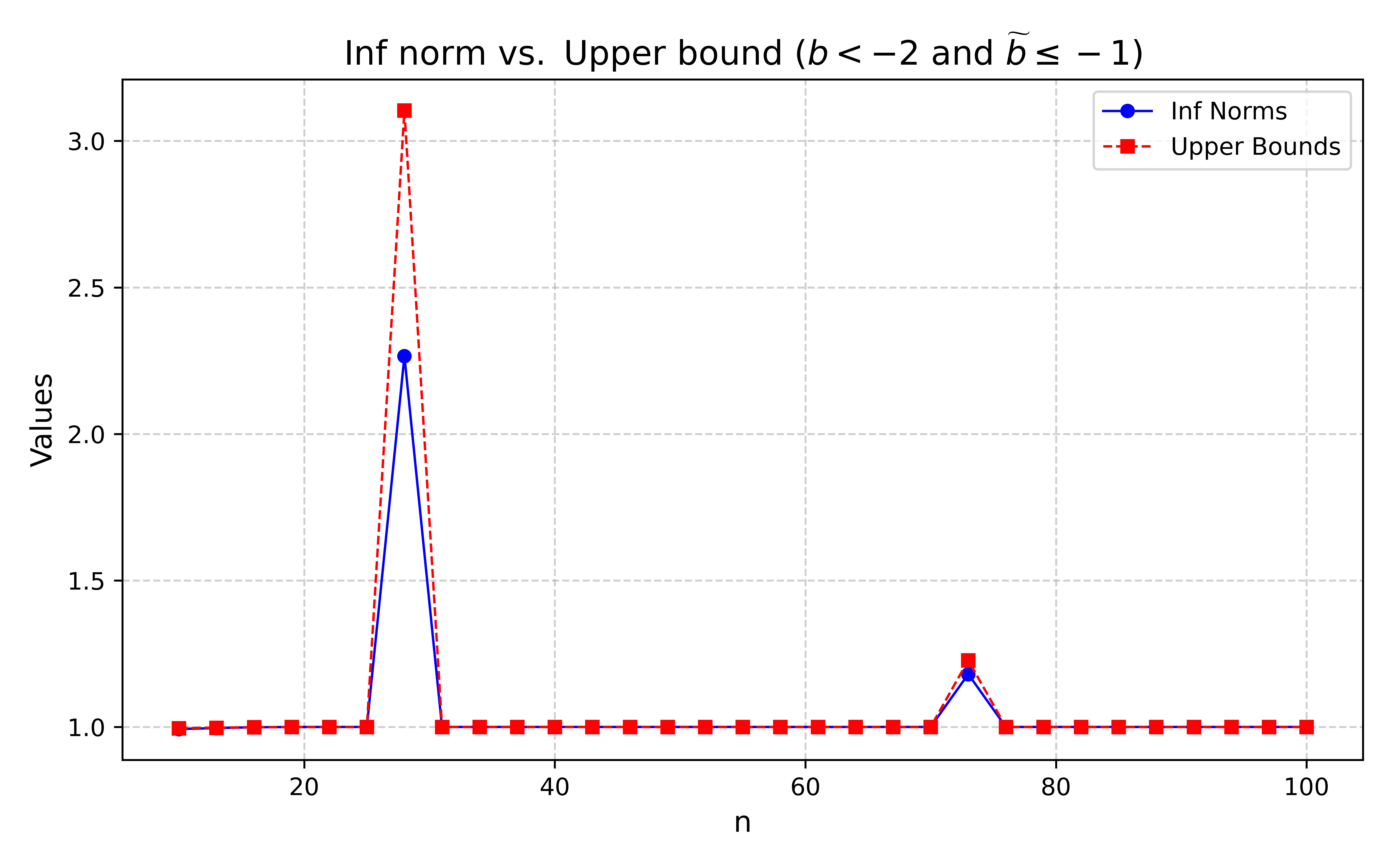}
        \caption{$\|\wtT^{-1}_{n}\|_{\infty}$ and upper bounds.}
        \label{fig5}
    \end{subfigure}
    \caption{Comparison of $\|\wtT^{-1}_{n}\|_{\infty}$ vs. Upper bounds when $b < -2$ and $\wtb \leq -1$.}
\end{figure}

\begin{figure}[htbp]
    \centering
    \begin{subfigure}{0.4\textwidth}
        \centering
        
        \label{tab4}
        \begin{tabular}{cccc}
            \toprule
            $n$ & $\widetilde{b}$ & $\|\wtT^{-1}_{n}\|_{\infty}$ & $\text{upper\_bound}$ \\
            \midrule
            10 &  4.57  &  0.973 &  3.826 \\
            13 &  1.29  &  0.995 &  10.071 \\
            16 &  0.81  &  1.352 &  15.470 \\
            19 &  3.41  &  1.000 &  4.578 \\
            22 &  4.76  &  1.000 &  3.754 \\
            25 &  2.38  &  1.000 &  5.926 \\ 
            28 & 6.85   &  1.000 &   3.161 \\ 
            \bottomrule
        \end{tabular}
        \caption{Comparison of $\|\wtT^{-1}_{n}\|_{\infty}$ and upper bounds.}
    \end{subfigure}%
    \hfill
    \begin{subfigure}{0.45\textwidth}
        \centering
        \includegraphics[width=\textwidth]{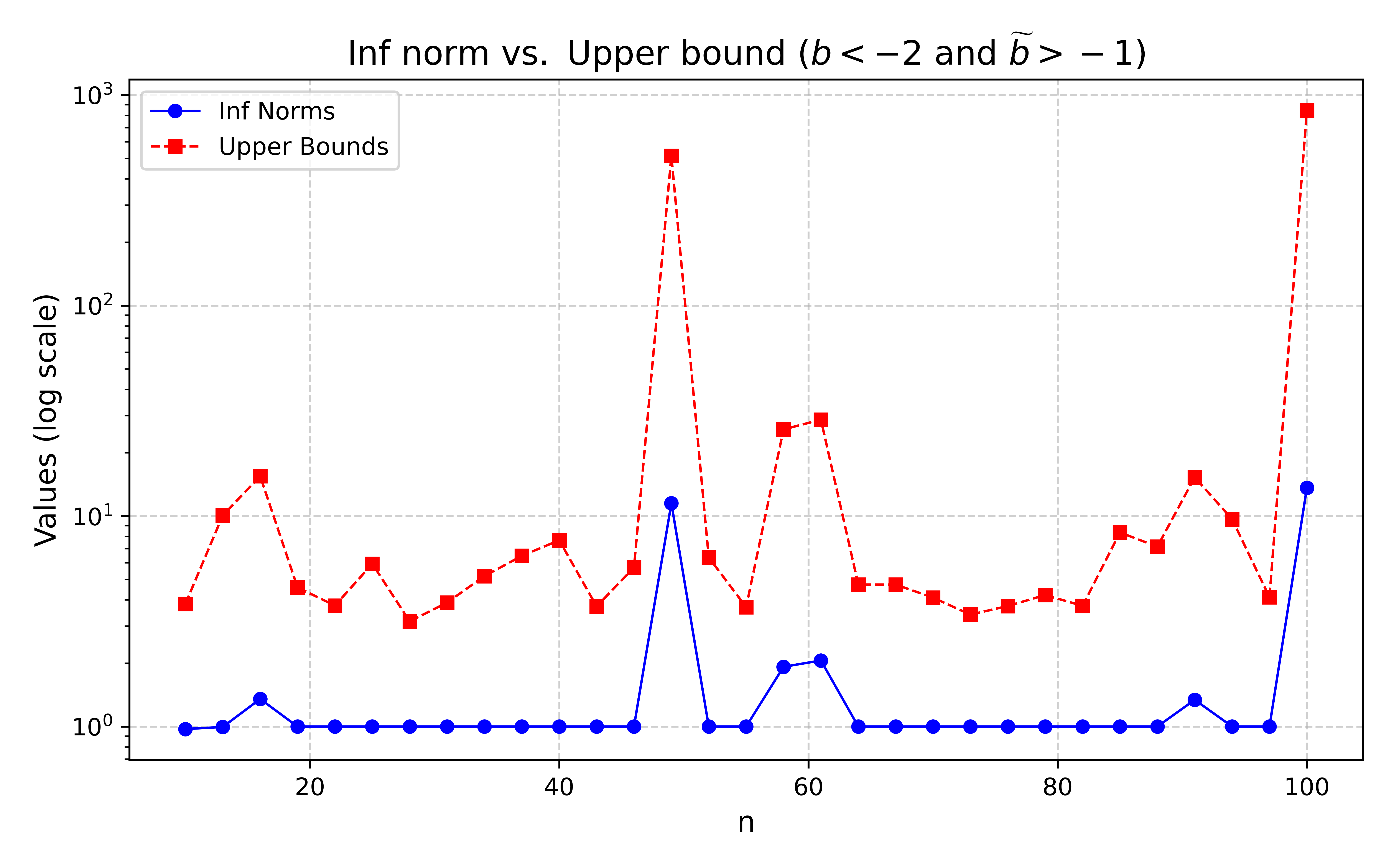}
        \caption{Log comparison of values.}
        \label{fig6}
    \end{subfigure}
    \caption{Comparison of $\|\wtT^{-1}_{n}\|_{\infty}$ vs. Upper bounds when $b < -2$ and $\wtb > -1$.}
\end{figure}

\FloatBarrier
\begin{table}[H]
    \centering
    \caption{Observed maximum convergence rate for Fisher's problem with $n = 20$, $b = \wtb = 4.0$, and $L = 2.0$. The expected rates are determined using upper bounds from Equation \eqref{upper_bound_1} and Theorem \ref{theor_upper_bound_2} applied to the norms of the inverse matrix.}
    \label{table5}
    \begin{tabular}{cccc}
        \toprule
        $k$ & $\textbf{Iterations}$ & $\textbf{Numerical rate}$ & $\textbf{Expected rate}$ \\
        \midrule
        1/2 & 4 & 0.0024 & 0.0025 \\
        1 & 4 & 0.0048 & 0.005 \\
        2 & 4 & 0.0097 & 0.01 \\
        4 & 5 & 0.0196 & 0.02 \\
        8 & 6 & 0.0391 & 0.04 \\
        16 & 7 & 0.0789 & 0.08 \\
        32 & 9 & 0.1564 & 0.16 \\
        \bottomrule
    \end{tabular}
\end{table}

\begin{table}[H]
    \centering
    \caption{Observed maximum convergence rate for Fisher's problem with $n = 50$, $b = \wtb = -4.0$, and $L = 1.0$. The expected rates are determined using upper bounds from Theorems \ref{theor_upper_bound_3} and \ref{theor_upper_bound_4} applied to the norms of the inverse matrix.}
    \label{table6}
    \begin{tabular}{cccc}
        \toprule
        $k$ & $\textbf{Iterations}$ & $\textbf{Numerical rate}$ & $\textbf{Expected rate}$ \\
        \midrule
        1 & 3 & 0.0001 & 0.0002 \\
        3 & 3 & 0.0003 & 0.0006 \\
        9 & 3 & 0.001 & 0.0018 \\
        27 & 4 & 0.0046 & 0.0054 \\
        81 & 5 & 0.0153 & 0.0162 \\
        273 & 6 & 0.0461 & 0.0486 \\
        729 & 8 & 0.137 & 0.1458 \\
        \bottomrule
    \end{tabular}
\end{table}

\section{Conclusion} \label{sec:concl}
The objective of this work is to present the explicit or the exact formula of bounds for the norm of inverse of some tridiagonal near-Toeplitz matrices. In addition, we provide the value of the traces and row sums of the inverse matrices. We show bounds  for $|b|>2$.
Theoretical results were validated by the numerical experiments. The findings suggest that these bounds could be useful for the fixed-point iterations in the convergence analysis.
Future research outcomes should explore scenarios where the Toeplitz part of tridiagonal near-Toeplitz matrices is weakly dominant, specifically when $|b| = 2$.


\FloatBarrier
\bibliographystyle{plain}
\bibliography{reference}

\begin{thebibliography}{10}

\bibitem{akaike1973block}
Hirotugu Akaike.
\newblock Block toeplitz matrix inversion.
\newblock {\em SIAM Journal on Applied Mathematics}, 24(2):234--241, 1973.

\bibitem{akansu2012toeplitz}
Ali~N Akansu and Mustafa~U Torun.
\newblock Toeplitz approximation to empirical correlation matrix of asset returns: A signal processing perspective.
\newblock {\em IEEE Journal of Selected Topics in Signal Processing}, 6(4):319--326, 2012.

\bibitem{allgower1973exact}
EL~Allgower.
\newblock Exact inverses of certain band matrices.
\newblock {\em Numerische Mathematik}, 21:279--284, 1973.

\bibitem{amanbek2020explicit}
Yerlan Amanbek, Zhibin Du, Yogi Erlangga, Carlos~M Da~Fonseca, Bakytzhan Kurmanbek, and Ant{\'o}nio Pereira.
\newblock Explicit determinantal formula for a class of banded matrices.
\newblock {\em Open Mathematics}, 18(1):1227--1229, 2020.

\bibitem{andjelic2021some}
Milica An{\dj}eli{\'c} and Carlos~M da~Fonseca.
\newblock Some determinantal considerations for pentadiagonal matrices.
\newblock {\em Linear and Multilinear Algebra}, 69(16):3121--3129, 2021.

\bibitem{bottcher1999convergence}
A~B{\"o}ttcher, S~Grudsky, A~Kozak, and B~Silbermann.
\newblock Convergence speed estimates for the norms of the inverses of large truncated toeplitz matrices.
\newblock {\em Calcolo}, 36:103--122, 1999.

\bibitem{diele1998use}
F~Diele and Luciano Lopez.
\newblock The use of the factorization of five-diagonal matrices by tridiagonal toeplitz matrices.
\newblock {\em Applied mathematics letters}, 11(3):61--69, 1998.

\bibitem{dow2002explicit}
Murray Dow.
\newblock Explicit inverses of toeplitz and associated matrices.
\newblock {\em ANZIAM Journal}, 44:E185--E215, 2002.

\bibitem{fischer1974fourier}
Dietrich Fischer, Gene Golub, Ole Hald, Carlos Leiva, and Olof Widlund.
\newblock On fourier-toeplitz methods for separable elliptic problems.
\newblock {\em Mathematics of Computation}, 28(126):349--368, 1974.

\bibitem{heinig2013algebraic}
Georg Heinig et~al.
\newblock {\em Algebraic methods for Toeplitz-like matrices and operators}, volume~13.
\newblock Birkh{\"a}user, 2013.

\bibitem{huang1997analytical}
Y~Huang and WF~McColl.
\newblock Analytical inversion of general tridiagonal matrices.
\newblock {\em Journal of Physics A: Mathematical and General}, 30(22):7919, 1997.

\bibitem{kurmanbek2022proof}
Bakytzhan Kurmanbek, Yerlan Amanbek, and Yogi Erlangga.
\newblock A proof of an{\dj}eli{\'c}-fonseca conjectures on the determinant of some toeplitz matrices and their generalization.
\newblock {\em Linear and Multilinear Algebra}, 70(8):1563--1570, 2022.

\bibitem{kurmanbek2021explicit}
Bakytzhan Kurmanbek, Yogi Erlangga, and Yerlan Amanbek.
\newblock Explicit inverse of near toeplitz pentadiagonal matrices related to higher order difference operators.
\newblock {\em Results in Applied Mathematics}, 11:100164, 2021.

\bibitem{kurmanbek2021inverse}
Bakytzhan Kurmanbek, Yogi Erlangga, and Yerlan Amanbek.
\newblock Inverse properties of a class of seven-diagonal (near) toeplitz matrices.
\newblock {\em Special Matrices}, 10(1):67--86, 2021.

\bibitem{lewis1982inversion}
Jerry~W Lewis.
\newblock Inversion of tridiagonal matrices.
\newblock {\em Numerische Mathematik}, 38:333--345, 1982.

\bibitem{luati2010spectral}
Alessandra Luati and Tommaso Proietti.
\newblock On the spectral properties of matrices associated with trend filters.
\newblock {\em Econometric Theory}, 26(4):1247--1261, 2010.

\bibitem{lv2008note}
Xiao-Guang Lv, Ting-Zhu Huang, and Jiang Le.
\newblock A note on computing the inverse and the determinant of a pentadiagonal toeplitz matrix.
\newblock {\em Applied Mathematics and Computation}, 206(1):327--331, 2008.

\bibitem{pan2015estimating}
Victor~Y Pan, John Svadlenka, and Liang Zhao.
\newblock Estimating the norms of random circulant and toeplitz matrices and their inverses.
\newblock {\em Linear algebra and its applications}, 468:197--210, 2015.

\bibitem{pozrikidis2014introduction}
Constantine Pozrikidis.
\newblock {\em An introduction to grids, graphs, and networks}.
\newblock Oxford University Press, USA, 2014.

\bibitem{reddi1984eigenvector}
SS~Reddi.
\newblock Eigenvector properties of toeplitz matrices and their application to spectral analysis of time series.
\newblock {\em Signal Processing}, 7(1):45--56, 1984.

\bibitem{schlegel1970explicit}
Peter Schlegel.
\newblock The explicit inverse of a tridiagonal matrix.
\newblock {\em mathematics of computation}, 24(111):665, 1970.

\bibitem{shitov2021determinants}
Yaroslav Shitov.
\newblock The determinants of certain (0, 1) toeplitz matrices.
\newblock {\em Linear Algebra and its Applications}, 618:150--157, 2021.

\bibitem{smith1985numerical}
Gordon~D Smith.
\newblock {\em Numerical solution of partial differential equations: finite difference methods}.
\newblock Oxford university press, 1985.

\bibitem{sweet1969recursive}
Roland~A Sweet.
\newblock A recursive relation for the determinant of a pentadiagonal matrix.
\newblock {\em Communications of the ACM}, 12(6):330--332, 1969.

\bibitem{tan2019explicit}
Linda~SL Tan.
\newblock Explicit inverse of tridiagonal matrix with applications in autoregressive modelling.
\newblock {\em IMA Journal of Applied Mathematics}, 84(4):679--695, 2019.

\bibitem{turkmen2002bounds}
Ramazan T{\"u}rkmen and Durmu{\c{s}} Bozkurt.
\newblock On the bounds for the norms of cauchy--toeplitz and cauchy--hankel matrices.
\newblock {\em Applied Mathematics and Computation}, 132(2-3):633--642, 2002.

\bibitem{usmani1994inversion}
Riaz~A Usmani.
\newblock Inversion of a tridiagonal jacobi matrix.
\newblock {\em Linear Algebra and its Applications}, 212(213):413--414, 1994.

\bibitem{wang2015explicit}
Chaojie Wang, Hongyi Li, and Di~Zhao.
\newblock An explicit formula for the inverse of a pentadiagonal toeplitz matrix.
\newblock {\em Journal of Computational and Applied Mathematics}, 278:12--18, 2015.

\bibitem{yamamoto1979inversion}
Tetsuro Yamamoto and Yasuhiko Ikebe.
\newblock Inversion of band matrices.
\newblock {\em Linear Algebra and Its Applications}, 24:105--111, 1979.

\bibitem{yueh2008explicit}
Wen-Chyuan Yueh and Sui~Sun Cheng.
\newblock Explicit eigenvalues and inverses of tridiagonal toeplitz matrices with four perturbed corners.
\newblock {\em the ANZIAM Journal}, 49(3):361--387, 2008.

\end{thebibliography}

\newpage
\appendix

\section{Proof of Lemma \ref{lem:gamm3}} \label{E}

Using Lemma~\ref{lem:gamm1},
$$
\frac{\gamma_{n+1} - \gamma_n - \gamma_1}{\gamma_n + \gamma_1} = \frac{\gamma_{n+1}}{\gamma_n + \gamma_1} - 1 = b - 1 - \frac{\gamma_{n-1} + b \gamma_1}{\gamma_n + \gamma_1}.
$$
Since $\gamma_n < b \gamma_{n-1}$ (Lemma~\ref{lem:gamm1}), 
$$
\frac{\gamma_{n+1} - \gamma_n - \gamma_1}{\gamma_n + \gamma_1} <  b - 1 - \frac{\gamma_{n-1} + b \gamma_1}{b \gamma_{n-1} + \gamma_1} = b -  1 - \frac{1}{b} \frac{\gamma_{n-1} + b \gamma_1}{\gamma_{n-1} + \frac{1}{b} \gamma_1} < b - 1 - \frac{1}{b} < b - 1,
$$
for $b > 2$. 
Now, let's have a look at
\begin{align*}
    \frac{\gamma_{n-1}+b\gamma_1}{\gamma_n + \gamma_1} = \frac{\gamma_{n-1}+\gamma_2}{\gamma_n + \gamma_1} &= \frac{r^{n-1}_{1} - r^{n-1}_{2} + r^2_{1} - r^2_{2}}{r^n_{1} - r^n_{2} + r_1 - r_2}\\
    &= \frac{\left(r^{\frac{n+1}{2}}_{1} - r^{\frac{n+1}{2}}_{2}\right) \left(r^{\frac{n-3}{2}}_{1} + r^{\frac{n-3}{2}}_{2}\right)}{\left(r^{\frac{n+1}{2}}_{1} - r^{\frac{n+1}{2}}_{2}\right) \left(r^{\frac{n-1}{2}}_{1} + r^{\frac{n-1}{2}}_{2}\right)}\\
    &= \frac{r^{\frac{n-3}{2}}_{1} + r^{\frac{n-3}{2}}_{2}}{r^{\frac{n-1}{2}}_{1} + r^{\frac{n-1}{2}}_{2}}.
\end{align*}
Let's look at the difference
\begin{align*}
    \frac{\gamma_{n-1} + b\gamma_1}{\gamma_n + \gamma_1} - \frac{\gamma_{n} + b\gamma_1}{\gamma_{n+1} + \gamma_1} &= \frac{r^{\frac{n-3}{2}}_{1} + r^{\frac{n-3}{2}}_{2}}{r^{\frac{n-1}{2}}_{1} + r^{\frac{n-1}{2}}_{2}} - \frac{r^{\frac{n-2}{2}}_{1} + r^{\frac{n-2}{2}}_{2}}{r^{\frac{n}{2}}_{1} + r^{\frac{n}{2}}_{2}}\\
    &= \frac{r^{\frac{3}{2}}_1 + r^{\frac{3}{2}}_2 - r^{\frac{1}{2}}_1 - r^{\frac{1}{2}}_2}{\left(r^{\frac{n-1}{2}}_{1} + r^{\frac{n-1}{2}}_{2} \right) \left(r^{\frac{n}{2}}_{1} + r^{\frac{n}{2}}_{2} \right)}\\
    &= \frac{r^2_1 + r_2 - r_1 - 1}{r^{\frac{1}{2}}_1\left(r^{\frac{n-1}{2}}_{1} + r^{\frac{n-1}{2}}_{2} \right) \left(r^{\frac{n}{2}}_{1} + r^{\frac{n}{2}}_{2} \right)}. 
\end{align*}
And note that
\begin{align*}
    r^2_1 + r_2 - r_1 - 1 &= \frac{b^2 - 2 + b\sqrt{b^2 - 4}}{2} - \sqrt{b^2 - 4} - 1 \\
    &= \frac{b^2 - 4}{2} + \frac{\sqrt{b^2 - 4}}{2}(b - 2) > 0.
\end{align*}
Therefore, 
$$\frac{\gamma_{n-1} + b\gamma_1}{\gamma_n + \gamma_1} > \frac{\gamma_{n} + b\gamma_1}{\gamma_{n+1} + \gamma_1},$$
which means
\begin{align*}
    \max_{n\geq 3}{\frac{\gamma_{n-1} + b\gamma_1}{\gamma_n + \gamma_1}} = \frac{2\gamma_2}{\gamma_3 + \gamma_1} = \frac{2\gamma_2}{b\gamma_2} = \frac{2}{b}.
\end{align*}
Also, we note that
\begin{align*}
    \frac{\gamma_{n-1}+b \gamma_1}{\gamma_n + \gamma_1} = \frac{r^{\frac{n-3}{2}}_1 + r^{\frac{n-3}{2}}_2}{r^{\frac{n-1}{2}}_{1} + r^{\frac{n-1}{2}}_{2}} > \frac{r^{\frac{n-3}{2}}_1 + r^{\frac{n-3}{2}}_2}{r_{1} \left( r^{\frac{n-3}{2}}_1 + r^{\frac{n-3}{2}}_2\right)} = \frac{1}{r_1} = r_2.
\end{align*}
Therefore
\begin{align*}
    \frac{\gamma_{n+1} -\gamma_n - \gamma_1}{\gamma_n + \gamma_1} = b - 1 - \frac{\gamma_{n-1}+b\gamma_1}{\gamma_n + \gamma_1} \geq b - 1 - \frac{2}{b}
\end{align*}
and
\begin{align*}
    \frac{\gamma_{n+1} -\gamma_n - \gamma_1}{\gamma_n + \gamma_1} < b - 1 - r_2 = r_1 - 1. 
\end{align*}

Putting all together, we have
$$
b-2 < b - 1 - \frac{2}{b} \leq \frac{\gamma_{n+1} - \gamma_n - \gamma_1}{\gamma_n + \gamma_1} < r_1 - 1 < b-1,
$$
which can be simplified to get the lemma.

\section{Proof of Lemma \ref{lem:compact_t}} \label{F}

As we know for $i \geq j$ from \eqref{eq:wtt}
\begin{align*}
    \wtt^{-1}_{i,j} = t^{-1}_{i, j} - \frac{\beta}{\Delta} \left[t^{-1}_{i, 1} (m_{1, 1} t^{-1}_{1, j} - m_{1, 2} t^{-1}_{n, j}) + t^{-1}_{i, n} (-m_{1, 2} t^{-1}_{1, j} + m_{1,1} t^{-1}_{n,j})\right]
\end{align*}
where $t^{-1}_{i, j} = \frac{\gamma_{j} \gamma_{n+1-i}}{\gamma_1 \gamma_{n+1}}$ for $i \geq j$ and $\beta = \wtb - b$ with 
\begin{align*}
    m_{1, 1} &= 1 + \beta t^{-1}_{n,n} = 1 + \frac{\beta \gamma_n}{\gamma_{n+1}}\\
    m_{1, 2} &= \beta t^{-1}_{n, 1} = \frac{\beta \gamma_1}{\gamma_{n+1}}
\end{align*}
Now, we will find each piece of $\wtt^{-1}_{i, j}$ and try to write it in a more compact way. 
\begin{align*}
    m_{1,1} t^{-1}_{1, j} - m_{1, 2} t^{-1}_{n, j} &= \left( 1 + \frac{\beta \gamma_n}{\gamma_{n+1}}\right) \frac{\gamma_{n+1-j}}{\gamma_{n+1}} - \frac{\beta \gamma_1}{\gamma_{n+1}} \frac{\gamma_j}{\gamma_{n+1}} \\
    &= \frac{\gamma_{n+1-j}}{\gamma_{n+1}} + \frac{\beta}{\gamma^2_{n+1}} (\gamma_n \gamma_{n+1-j} - \gamma_1 \gamma_j)\\
    &= \frac{\gamma_{n+1-j}}{\gamma_{n+1}} + \frac{\beta \gamma_{n-j}}{\gamma_{n+1}}
\end{align*}
The last equality due to 
\begin{align} \label{id_1}
    \gamma_n \gamma_{n+1-j} - \gamma_1 \gamma_j 
    &= r^{2n+1-j}_{1} - r^{j-1}_{1} - r^{j-1}_{2} + r^{2n+1-j}_{2} - r^{j+1}_{1} + r^{j-1}_{1} + r^{j-1}_{2} - r^{j+1}_{2} \notag \\ 
    &= r^{2n+1-j}_{1} - r^{j+1}_{1} - r^{j+1}_{2} + r^{2n+1-j}_{2} \notag \\  
    &= \gamma_{n+1} \gamma_{n-j}.
\end{align}
In the same way, we get
\begin{align*}
    -m_{1, 2} t^{-1}_{1, j} + m_{1,1} t^{-1}_{n,j} &= -\frac{\beta \gamma_1}{\gamma_{n+1}} \frac{\gamma_{n+1-j}}{\gamma_{n+1}} + \left( 1 + \frac{\beta \gamma_n}{\gamma_{n+1}}\right) \frac{\gamma_j}{\gamma_{n+1}}\\
    &= \frac{\gamma_j}{\gamma_{n+1}} + \frac{\beta}{\gamma^2_{n+1}} (\gamma_n \gamma_j - \gamma_1 \gamma_{n+1-j})\\
    &= \frac{\gamma_j}{\gamma_{n+1}} + \frac{\beta \gamma_{j-1}}{\gamma_{n+1}}.
\end{align*}
The $\Delta$ part is 
\begin{align*}
    \Delta = m^2_{1, 1} - m^{2}_{1, 2} &= \left(1 + \frac{\beta (\gamma_{n} + \gamma_1)}{\gamma_{n+1}} \right) \left(1 + \frac{\beta (\gamma_n - \gamma_1)}{\gamma_{n+1}} \right)\\
    &= \frac{(\gamma_{n+1} + \beta \gamma_n)^2 - (\beta \gamma_1)^2}{\gamma^2_{n+1}}
\end{align*}
Therefore, 
\begin{align} \label{eq2_wtt_init}
    \wtt^{-1}_{i, j} &= \frac{\gamma_j \gamma_{n+1-i}}{\gamma_1 \gamma_{n+1}} - \frac{\beta}{\Delta \gamma^2_{n+1}}\left( \gamma_{n+1-i} (\gamma_{n+1-j} + \beta \gamma_{n-j}) + \gamma_i (\gamma_j + \beta \gamma_{j-1}) \right) \notag\\
    &= \frac{\gamma_j \gamma_{n+1-i}}{\gamma_1 \gamma_{n+1}} - \beta\frac{\gamma_{n+1-i} (\gamma_{n+1-j} + \beta \gamma_{n-j}) + \gamma_i (\gamma_j + \beta \gamma_{j-1})  }{(\gamma_{n+1} + \beta \gamma_n)^2 - (\beta \gamma_1)^2}. 
\end{align}
Now, let us consider the following parts of the $\wtt^{-1}_{i, j}$, so we can simply further
\begin{align*}
    P_{i, j} &:=\gamma_j ((\gamma_{n+1} + \beta \gamma_n)^2 - (\beta \gamma_1)^2) - \beta\gamma_1 \gamma_{n+1} (\gamma_{n+1-j} + \beta \gamma_{n-j}) \quad \text{[using $\gamma^2_{n} - \gamma^2_{1} = \gamma_{n+1}\gamma_{n-1}$]}\\ 
    &= \gamma_j \gamma^2_{n+1} + \beta (2 \gamma_j \gamma_n \gamma_{n+1} - \gamma_1 \gamma_{n+1}\gamma_{n+1 -j}) + \beta^2 (\gamma_j \gamma_{n+1}\gamma_{n-1} - \gamma_1\gamma_{n+1}\gamma_{n-j})\\
    &= \gamma_{n+1}\left[\gamma_j \gamma_{n+1} + \beta (2 \gamma_j \gamma_n - \gamma_1 \gamma_{n+1-j}) + \beta^2 (\gamma_j \gamma_{n-1} - \gamma_1 \gamma_{n-j})\right]\\
    &= \gamma_{n+1}\left[\gamma_j \gamma_{n+1} + \beta(\gamma_j \gamma_n + \gamma_{n+1}\gamma_{j-1}) + \beta^2 \gamma_{n} \gamma_{j-1}  \right].
\end{align*}
The last equality is due to the previous shown identity $\gamma_n \gamma_{n+1-j} - \gamma_1 \gamma_j = \gamma_{n+1}\gamma_{n-j}$ with $j \rightarrow n+1-j$, so we get $\gamma_j \gamma_n  - \gamma_1 \gamma_{n+1-j} = \gamma_{n+1}\gamma_{j-1}$ and substituting $n \rightarrow n-1$ to the new one, we get $\gamma_j \gamma_{n-1} - \gamma_1 \gamma_{n-j} = \gamma_n \gamma_{j-1}$.
\begin{align*}
    Q_{i, j} &:= \frac{\gamma_{n+1-i} \cdot P_{i, j}}{\gamma_{n+1}} - \beta \gamma_1 \gamma_i (\gamma_j + \beta \gamma_{j-1}) \\
    &= \gamma_{n+1-i}\left[\gamma_j \gamma_{n+1} + \beta(\gamma_j \gamma_n + \gamma_{n+1}\gamma_{j-1}) + \beta^2 \gamma_{n} \gamma_{j-1}  \right] - \beta\gamma_1 \gamma_i (\gamma_j + \beta \gamma_{j-1})\\
    &= \gamma_{n+1-i}\gamma_{n+1}(\gamma_j + \beta \gamma_{j-1})+ \beta \gamma_j (\gamma_{n+1-i} \gamma_n - \gamma_1 \gamma_{i}) + \beta^2 \gamma_{j-1}(\gamma_{n+1-i}\gamma_n  - \gamma_1 \gamma_{i})\\
    &= \gamma_{n+1-i}\gamma_{n+1}(\gamma_j + \beta \gamma_{j-1}) + \beta \gamma_j \gamma_{n+1} \gamma_{n-i} + \beta^2 \gamma_{j-1} \gamma_{n+1}\gamma_{n-i} \quad \text{[by identity]}\\
    &= \gamma_{n+1-i}\gamma_{n+1}(\gamma_j + \beta \gamma_{j-1}) + \beta \gamma_{n+1} \gamma_{n-i} (\gamma_j + \beta \gamma_{j-1})\\
    &= \gamma_{n+1}(\gamma_{n+1-i} + \beta \gamma_{n-i})(\gamma_j + \beta \gamma_{j-1}).
\end{align*}
Therefore for $i \geq j$, 
\begin{align}\label{eq2_wtt}
    \wtt^{-1}_{i, j} &= \frac{\gamma_{n+1-i}}{\gamma_1 \gamma_{n+1} ((\gamma_{n+1} + \beta \gamma_n)^2 - (\beta \gamma_1)^2)} P_{i, j} - \frac{\beta \gamma_i (\gamma_j + \beta \gamma_{j-1})}{(\gamma_{n+1} + \beta \gamma_n)^2 - (\beta \gamma_1)^2} \notag\\
    &= \frac{Q_{i, j}}{\gamma_1((\gamma_{n+1} + \beta \gamma_n)^2 - (\beta \gamma_1)^2)} \notag \\
    &= \frac{\gamma_{n+1}}{\gamma_1 ((\gamma_{n+1} + \beta \gamma_n)^2 - (\beta \gamma_1)^2)} [\gamma_{n+1-i} + \beta \gamma_{n-i}][\gamma_j + \beta \gamma_{j-1}] \notag \\
    &= C (\gamma_{n+1-i} + \beta \gamma_{n-i})(\gamma_j + \beta \gamma_{j-1})
\end{align}
where $C:= \frac{\gamma_{n+1}}{\gamma_1 ((\gamma_{n+1} + \beta \gamma_n)^2 - (\beta \gamma_1)^2)}$.

\section{Proof of Lemma \ref{lem:sum_squared_gammas}}\label{G}
With $r_1 r_2 = 1$,
\begin{align}
&\sum_{i=1}^n \left\{ \gamma_i^2 + \gamma_{n+1-i}^2 \right\} = \notag \\ 
&= \sum_{i=1}^n \left\{ r_1^{2(n+1)} \left( \frac{1}{r_1^2} \right)^i + r_2^{2(n+1)} \left( \frac{1}{r_2^2} \right)^i + (r_1^2)^i + (r_2^2)^i - 4 \right\} \notag \\
&=  r_1^{2(n+1)}   \left\{ \sum_{i = 0}^n \left( \frac{1}{r_1^2} \right)^i  - 1 \right\} + r_2^{2(n+1)} \left\{  \sum_{i=0}^n  \left( \frac{1}{r_2^2} \right)^i  - 1 \right\} + \sum_{i=0}^n (r_1^2)^i - 1 + \sum_{i=0}^n (r_2^2)^i - 1 - \sum_{i=1}^n 4 \notag \\
&= r_1^{2(n+1)} \left\{ \displaystyle \frac{1- \left(\frac{1}{r_1^2}\right)^{n+1}}{1-\frac{1}{r_1^2}}  - 1\right\} +  r_2^{2(n+1)} \left\{ \displaystyle \frac{1- \left(\frac{1}{r_2^2}\right)^{n+1}}{1-\frac{1}{r_2^2}}  - 1\right\} + \frac{1 - (r_1^2)^{n+1}}{1-r_1^2} +  \frac{1 - (r_2^2)^{n+1}}{1-r_2^2} - 4n - 2 \notag \\
&= \frac{r_1^{2(n+1)} - 1}{r_1^2 - 1}(r_1^2 + 1) + \frac{r_2^{2(n+1)}-1}{r_2^2 - 1} (r_2^2+1) -\left( r_1^{2(n+1)} + r_2^{2(n+1)} \right) - 4n - 2 \notag \\
&= \frac{(r_1^{2(n+1)} - 1)(r_1^2 +1)(r_2^2-1) + (r_2^{2(n+1)} - 1)(r_2^2 + 1)(r_1^2-1)}{(r_1^2-1)(r_2^2 -1)} - \left( r_1^{2(n+1)} + r_2^{2(n+1)} \right) - 4n - 2 \notag \\
&= \frac{(r_1^{2(n+1)} - 1)(r_2^2 - r_1^2 ) + (r_2^{2(n+1)} - 1)(r_1^2 - r_2^2)}{2 - (r_1^2 + r_2^2)} - \left( r_1^{2(n+1)} + r_2^{2(n+1)} \right) - 4n - 2 \notag \\
&= - \frac{\gamma_2 (r_1^{2(n+1)} - r_2^{2(n+1)})}{2 - (r_1^2 + r_2^2)} - \left( r_1^{2(n+1)} + r_2^{2(n+1)} \right) - 4n - 2,
\end{align}
because $\gamma_2 := r_1^2 - r_2^2$. Notice that for $k = 1,2, \dots$, 
\begin{itemize}
\item $r_1^{2k} + r_2^{2k} = (r_1^k - r_2^k)^2 + 2(r_1r_2)^k  = \gamma_k^2 + 2$, and
\item $(r_1^{2k} - r_2^{2k}) = (r_1^k + r_2^k)(r_1^k - r_2^k) = \gamma_k \left( \gamma^2_{\frac{k}{2}} + 2 \right)$.
\end{itemize}
Using the above relations, $\gamma_1 = \sqrt{b^2-4}$, and $\gamma_2 = b\sqrt{b^2-4}$, we get the result in the lemma.

\end{document}